\documentclass[11pt]{amsart}
\usepackage[cp1251]{inputenc}

\usepackage{amssymb,amsmath,amscd,color}

\theoremstyle{plain}
\newtheorem{thm}{Theorem}

\newtheorem{lem}{Lemma}

\theoremstyle{definition}
\newtheorem{defn}{Definition}

\theoremstyle{remark}

\newtheorem*{remark*}{Remark}

\newcommand{\cE}{{\mathcal{E}}}

\newcommand{\cH}{{\mathcal{H}}}

        \newcommand{\field}[1]{{\mathbb{#1}}}
        \newcommand{\NN}{\field{N}}
        \newcommand{\ZZ}{\field{Z}}
        
        \newcommand{\RR}{\field{R}}
        \newcommand{\CC}{\field{C}}

\newcommand{\supp}{\operatorname{supp}}

\allowdisplaybreaks

\begin{document}

\title[Trace formula for the magnetic Laplacian]{Trace formula for the magnetic Laplacian}

\author[Y. A. Kordyukov]{Yuri A. Kordyukov}
\address{Institute of Mathematics, Ufa Federal Research Centre, Russian Academy of Sciences, 112~Chernyshevsky str., 450008 Ufa, Russia and Novosibirsk State University, Pirogova st 1, 630090, Novosibirsk, Russia} \email{yurikor@matem.anrb.ru}

\author[I. A. Taimanov]{Iskander A. Taimanov}
\address{Sobolev Institute of Mathematics, 4 Acad. Koptyug avenue,
and Novosibirsk State University, Pirogova st 1, 630090, Novosibirsk, Russia}
\email{taimanov@math.nsc.ru}

\thanks{This work was supported by the Laboratory of Topology and Dynamics, Novosibirsk State University (contract no. 14.Y26.31.0025 with the Ministry of Education and Science of the Russian Federation).}



\begin{abstract}
The Guillemin-Uribe trace formula is a semiclassical version of the Selberg trace formula and  the more general Duistermaat-Guillemin formula for elliptic operators on compact manifolds, which reflects the dynamics of magnetic geodesic flows in terms of eigenvalues of a natural differential operator  (the magnetic Laplacian) associated with the magnetic field. In the present paper, we give a survey of basic notions and results related with the Guillemin-Uribe trace formula and provide concrete examples of its computation for two-dimensional constant curvature surfaces with constant magnetic fields and for the Katok example.
\end{abstract}

\date{}

 \maketitle
\tableofcontents
\section{Introduction}

The Selberg trace formula \cite{Selberg}, which plays the fundamental role in modern number theory, for a function $h\in C^\infty_c(\RR)$ and a compact hyperbolic surface $X$, expresses the trace of the operator $h(\Delta)$, where $\Delta$ is the Laplace-Beltrami operator on $X$, in terms of the volume of $X$ and the lengths of closed geodesics on it. Numerous applications of the formula and its generalizations in the theory of dynamical systems are related with this. For the circle, it turns into the Poisson summation formula, and, for the two-sphere of constant curvature, the formula for the trace of  $h(\Delta)$ has a natural interpretation in terms of the area of the sphere and closed geodesics, all of them being degenerate in this case 
(see, for instance, \cite{Marklof}).

The development of the variational theory of closed magnetic geodesics having its origins in the papers of S. P. Novikov \cite{Novikov-S,Novikov1982} gave rise to many remarkable results, which we will not dwell on here. In this paper, we will consider the question whether there exist analogs of the trace formula for magnetic geodesic flows, reflecting the dynamics of the flows in terms of the eigenvalues of some natural differential operator (the magnetic Laplacian) associated with the magnetic field.

The magnetic geodesic flow on an $n$-dimensional manifold $M$ is determined by a Riemannian metric $g$ and a closed $2$-form
$F$.  Namely, it is the following Hamiltonian system on the cotangent bundle $X=T^*M$ of $M$. Denote by $\Omega_0$ the canonical symplectic form on $T^*M$:
\[
\Omega_0=\sum_{j=1}^ndp_j\wedge dx^j.
\]
The magnetic flow $\Phi^t : T^*M\to T^*M$ associated with $(g,F)$ is the Hamiltonian flow given by the Hamiltonian 
\begin{equation}\label{e:1.2a}
\cH(x,p)=\frac{1}{2}|p|^2_{g^{-1}}=\frac{1}{2}\sum_{j,k=1}^ng^{jk}p_jp_k,
\end{equation}
with respect to the twisted symplectic form on $T^*M$:
\begin{equation}\label{e:1.3}
\Omega=\Omega_0+\pi^*_M F.
\end{equation}
Here $\pi_M :T^*M\to M$ denotes the canonical projection. In local coordinates, this expression is written as
\[
\Omega=\sum_{j=1}^{n}dp_j\wedge dx^j+\sum_{j,k=1}^n F_{jk}dx^j\wedge dx^k,
\]
where
\[
F=\sum_{j,k=1}^n F_{jk}dx^j\wedge dx^k.
\]

Hamilton's equations given by an arbitrary Hamiltonian $H$ with respect to the form $\Omega$ are written as 
\begin{equation}\label{e:1.4}
\frac{dx^j}{dt}=\frac{\partial H}{\partial p_j},\quad \frac{dp_j}{dt}=-\frac{\partial H}{\partial x^j}+\sum_{k=1}^n F_{jk}\frac{\partial H}{\partial p_k}, \quad j=1,\ldots,n.
\end{equation}
In particular, if $H$ is given by \eqref{e:1.2a}, we get a Hamiltonian system defining the magnetic flow $\Phi^t$:
\begin{equation}\label{e:1.5}
\frac{dx^j}{dt}= p^j,\quad \frac{dp_j}{dt}=-\frac{1}{2}\sum_{k,\ell=1}^n\frac{\partial g^{k\ell}}{\partial x^j}p_kp_\ell +\sum_{k=1}^nF_{jk}p^k, \quad j=1,\ldots,n,
\end{equation}
where $p^j=\sum_{k=1}^ng^{jk}p_k$.

Let $J : TM\to TM$ be the skew-symmetric operator such that
\[
F(u,v)=g(Ju,v), \quad u,v\in TM.
\]
If $(x(t),\xi(t))=\Phi^t(x,\xi)$ is a trajectory of the magnetic flow $\Phi$, then its projection to $M$ satisfies the second order differential equation 
\begin{equation}\label{e:magnetic}
\nabla^{TM}_{\dot{x}}\dot{x}=J[\dot{x}],
\end{equation}

It is easy to see that the restriction of the magnetic geodesic flow to different energy levels $\{\cH(x,p) = \cE = \mathrm{const}\}$ are trajectory non-isomorphic, as opposed to the case of geodesic flow. Therefore, we will always consider the magnetic geodesic flow at a fixed energy level.

A natural candidate for the magnetic analog of the Selberg trace formula is the Guillemin-Uribe formula \cite{Gu-Uribe89} (see the formula  (\ref{traceformula}) in Theorem \ref{t:varphi}), which can be applied in the case when the magnetic field form $F$ satisfies the quantization condition
\begin{equation}\label{e:prequantum}
[F]\in H^2(M,2\pi \mathbb Z).
\end{equation}
In this case, $F$ equals (up to multiplication by $i$) the curvature form of some Hermitian bundle $L$ on $M$. This fact allows us to construct the magnetic Laplacian associated with $F$, for which the Guillemin--Uribe formula holds.

Observe that there are another methods of constructing trace formulas for magnetic Laplacians \cite{CGO,HP,ColinDeV12}, which, however, are not related with restrictions of the flow to fixed energy levels.

Let us introduce some notions necessary for the exposition.

Let $(M,g)$ be a compact Riemannian manifold of dimension $n$ and let $(L,h^L)$ be a Hermitian line bundle on $M$ with a Hermitian connection 
\[
\nabla^L : C^\infty(M,L)\to C^\infty(M,T^*M\otimes L).
\] 
The curvature form of $\nabla^L$ is given by $R^L=(\nabla^L)^2$. We will assume that it is related with the magnetic field form $F$ by
\begin{equation}\label{e:1.2}
F = iR^L.
\end{equation}
It is well known that, if $F$ satisfies the quantization condition \eqref{e:prequantum}, then such a Hermitian line bundle $(L,h^L)$ with Hermitian connection $\nabla^L$ exists.

The Riemannian metric on $X$ and the Hermitian structure on $L$ allows us
to define inner products on $C^\infty(M,L)$ and $C^\infty(M,T^*M\otimes L)$ and the adjoint operator
\[
(\nabla^L)^* : C^\infty(M,T^*M\otimes L)\to C^\infty(M,L).
\]
The magnetic Laplacian is a second order differential operator acting on $C^\infty(M,L)$ by
\[
\Delta^{L} = (\nabla^L)^*\nabla^L.
\]
This operator can be considered as the Bochner Laplacian associated with the Hermitian line bundle $L$.

Let $\nabla^{TM}$ be the Levi-Civita connection of $g$. If $\{e_j\}_{j=1,\ldots,n}$ is a local orthonormal frame in the tangent bundle $TM$, then the operator $\Delta^{L}$ is given by
\begin{equation}\label{e:1.1}
\Delta^{L}=-\sum_{j=1}^n\left[(\nabla^{L}_{e_j})^2-\nabla^{L}_{\nabla^{TM}_{e_j}e_j} \right].
\end{equation}

Choose local coordinates $(x^1,\ldots,x^n)$ on an open subset $U$ of $M$. Assume that the Hermitian line bundle $L$ is trivial on $U$, i.e. 
\[
L\left|_U\right.\cong U\times \CC\ \text{and}\ |(x,z)|_{h^L}=|z|,\quad (x,z)\in U\times \CC.
\] 
Then the covariant derivative $\nabla^L$ can be written as
\[
\nabla^L = d-i A : C^\infty(U)\to C^\infty(U,T^*U),
\]
where $A= \sum_{j=1}^nA_j(x)\,dx^j$ is a real-valued connection 1-form (the magnetic potential).
It is easy to check that $R^L=-idA$ and $F$ coincides with the the magnetic field 2-form  
\[
B=dA=\sum_{j<k}B_{jk}\,dx^j\wedge dx^k, \quad
B_{jk}=\frac{\partial A_k}{\partial x^j}-\frac{\partial
A_j}{\partial x^k}.
\]
Let us write the matrix of the Riemannian metric $g$ as
$g(x)=(g_{j\ell}(x))_{1\leq j,\ell\leq n}$, its inverse as
$g(x)^{-1}=(g^{j\ell}(x))_{1\leq j,\ell\leq n}$
and denote $|g(x)|=\det(g(x))$, then $\Delta^{L}$ takes the form
\[
\Delta^{L}=-\frac{1}{\sqrt{|g(x)|}}\sum_{1\leq j,\ell\leq n}\left(\frac{\partial}{\partial x^j}-iA_j(x)\right) \left[\sqrt{|g(x)|}
g^{j\ell}(x) \left(\frac{\partial}{\partial
x^\ell}-iA_\ell(x)\right)\right].
\]

Observe that, if two Hermitian connections on the Hermitian line bundle $(L, h^L)$ are gauge equivalent, then the corresponding magnetic Laplacians are unitary equivalent and their spectra are the same. Therefore, if the manifold $M$ is simply connected, then the eigenvalues of the magnetic Laplacian $\Delta^L$ depend only on the magnetic field form $F$.

For any $N\in \NN$, consider the $N$th tensor power $L^N=L^{\otimes N}$ of the line bundle $L$. Denote by $\Delta^{L^N}$ the corresponding magnetic Laplacian on $C^\infty(M,L^N)$:
\begin{multline}\label{e:DLp}
\Delta^{L^N} =-\frac{1}{\sqrt{|g(x)|}}\sum_{1\leq j,\ell\leq n}\left(\frac{\partial}{\partial x^j}-iNA_j(x)\right)\times \\ \times  \left[\sqrt{|g(x)|}
g^{j\ell}(x) \left(\frac{\partial}{\partial
x^\ell}-iNA_\ell(x)\right)\right].
\end{multline}
Let $\{\nu_{N,j}, j=0,1,2,\ldots\}$ be the eigenvalues of the operator $\Delta^{L^N}$ taken with multiplicities. Put
\begin{equation}\label{e:def-lambda}
\lambda_{N,j}=\sqrt{\nu_{N,j}+N^2}.
\end{equation}

Following the notations of \cite{Gu-Uribe89}, we denote by $H(x,p)$ another Hamiltoinian related with $\cH$ by the formula
$$
H = \sqrt{2\cH+1},
$$
that implies the relation between the (physical) energy levels $\cE$ associated with the Hamiltonian $\cH$ and the (formal) energy levels $E$ associated with the Hamiltonian $H$:
$$
\cE = \frac{1}{2}\sum_{j,k=1}^ng^{jk}p_jp_k = \frac{E^2-1}{2}.
$$
For the sake of brevity and following \cite{Gu-Uribe89}, we will also use notation 
$$
c = \sqrt{E^2-1} = \sqrt{2\cE}.
$$

Fix $E>1$. It turns out that the asymptotic distribution of $\lambda_{N,j}$ in intervals of the form $(EN-\tau,EN+\tau)$ with an arbitrary $\tau>0$ is related with the geometry of the corresponding magnetic flow at the energy level $E$. This relationship is expressed by the trace formula proved by V. Guillemin and A. Uribe in \cite{Gu-Uribe89}, which describes the asymptotic behavior as $N\to\infty$ of the sequence
\begin{equation}\label{e:Yp}
Y_N(\varphi)=\sum_{j=0}^{\infty}\varphi(\lambda_{N,j}-EN),\quad N\in \mathbb N,
\end{equation}
defined by an arbitrary function $\varphi\in \mathcal S(\RR)$.

The Guillemin-Uribe trace formula is a semiclassical version of the Selberg trace formula and  the more general Duistermaat-Guillemin formula for elliptic operators on compact manifolds in the spirit of the semiclassical trace formula, first suggested by M. Gutzwiller in \cite{Gutzwilier} and mathematically rigorously proved in \cite{Meinrenken92,Meinrenken94,Paul-Uribe95,CRM99,Sj-Zworski}  (see also \cite{Uribe}). Unlike them, the Guillemin-Uribe trace formula is less known and scarcely considered by the experts in mathematical physics and dynamical systems. In \cite{Gu-Uribe90}, this formula was extended to the case of  the quantum Hamiltonian, describing the motion of a charged particle in an external Yang-Mills field with an arbitrary gauge group. 

In the present paper, after the exposition of the Guillemin-Uribe formula in Section 2, we give in Section 3 concrete examples of its computation for two-dimensional constant curvature surfaces with constant magnetic fields and for the Katok example. These computations show that the trace formula reflects some dynamical effects. However, for now, it is too early to talk about its broad applicability in the theory of dynamical systems.
Let us make a couple of remarks on this issue.

 {\sc Remark 1} (Ma\~{n}\'{e} level).  It was clear from the early 1980s that one should use different methods for large and small energy levels to prove the existence of periodic magnetic geodesics  
 \cite{Novikov1982,Taimanov1983}. In the case when the magnetic field is given by an exact form, the usual Morse theory can be applied for large energies, whereas for small energies (the case of ``strong'', with respect to the energy level, magnetic field), one should use completely different methods \cite{NT,Taimanov1,Taimanov2}. In the case when the magnetic field is not exact, an analog of ``strong'' magnetic fields was introduced in \cite{Taimanov3}. One can apply to such fields an approach suggested in \cite{Taimanov2}. The exposition of the state of the art in the variational theory of closed magnetic geodesics on surfaces can be found in \cite{A1,A2} (see also the references therein). In \cite{CMP}, it was established that, for exact magnetic fields on two-dimensional surfaces, the energy levels, with respect to which the magnetic field is ``strong'' or ``weak'', are separated by one constant, which equals the Ma\~{n}\'{e} level of this Hamiltonian system. The notion of the Ma\~{n}\'{e} level plays an important role in the theory of dynamical systems \cite{CIPP,CFP,Herbrich}.
\medskip\par
{\sl The trace formula applied to the magnetic geodesic flow associated with the hyperbolic metric
$y^{-2}R^2(dx^2+dy^2)$ and the magnetic field $F = y^{-1}dx \wedge dy$ is defined only for the energies below the Ma\~{n}\'{e} level $\cE_M = 1/(2R)$ and can't be extended analytically at this level and above it (see Theorem \ref{t:hyper} below)}.

Above this energy level, the flow becomes to be Anosov and is conjugate to the geodesic flow on the constant negative curvature surface. An analog of the classical Selberg trace formula holds for it \cite{CGO} (it is exactly a physical interpretation of the classical Maass-Selberg trace formula  for the Laplacian on automorphic forms \cite{Hejhal, Venkov}). It doesn't reflect the dynamics of the magnetic flow at fixed energy levels.
\medskip\par
{\sc Remark 2} (The Katok example).
The Katok example \cite{Katok}, demonstrating that the KAM theorem fails for degenerate Hamiltonians, can be represented in the form of the magnetic geodesic flow on the two-sphere at a fixed energy level \cite{Rademacher04}. Actually, there is a one-parameter family of flows, depending on a parameter $\varepsilon$, where $0 \leq \varepsilon <1$,
moreover: 1) if $\varepsilon = 0$, we have the usual geodesic flow on the sphere of constant curvature; 2) if $\varepsilon >0$, then, for rational $\varepsilon$, all trajectories are closed, and, for irrational $\varepsilon$, the flow has only two closed trajectories, which are the equator of the sphere, running in two opposite directions.
Note that these equators are trajectories for all values of $\varepsilon$.
\medskip\par
{\sl The trace formula for these flows for irrational $\varepsilon$ splits into two components, each of them is naturally identified with the contribution of one of equatorial trajectories. The coefficients of the asymptotic expansions are spectral invariants of the magnetic Laplacians. They can be considered as analogs of the wave invariants of closed geodesics.}

In the case of closed geodesics, the wave invariants \cite{G,Z1,Z2,GP} are expressed in terms of the coefficients of the quantum Birkhoff normal form for the wave operator in a neighborhood of the closed geodesic, that allows one to compute them in terms of invariants of the Riemannian metric and dynamic invariants of closed geodesics.

{\sl For closed magnetic geodesics, similar invariants haven't been considered yet.}

 \section{The trace formula}\label{s:trace}

\subsection{Classical dynamics}
Let $(M,g)$ be a compact Riemannian manifold of dimension $n$ and let $(L,h^L)$ be a Hermitian line bundle on $M$ with a Hermitian connection $\nabla^L$. Fix $E>1$. As mentioned in Introduction, we will consider the Hamiltonian flow $\phi$ on the cotangent bundle $X=T^*M$ of $M$ equipped with the twisted symplectic form $\Omega$ defined by the Hamiltonian $H$ given by
\begin{equation}\label{e:H}
H(x,p)=(2\cH(x,p)+1)^{1/2}=\left(\sum_{j,k=1}^ng^{jk}p_jp_k+1\right)^{1/2}.
\end{equation}
The corresponding Hamiltonian system has the form:
\begin{equation}\label{e:1.6}
\frac{dx^j}{dt}=\frac{p^j}{H},\quad \frac{dp_j}{dt}=-\frac{1}{2H}\sum_{k,\ell=1}^n\frac{\partial g^{k\ell}}{\partial x^j}p_kp_\ell + \frac{1}{H}\sum_{k=1}^n F_{jk}p^k, \quad j=1,\ldots,n.
\end{equation}
It is easy to see that the solutions of this system with a given energy level  $H=E$ coincide (up to a time change, depending on $E$) with the solutions of  \eqref{e:1.5}.

The trace formula (see Theorem \ref{t:varphi} below) is derived under an additional assumption on the set of periodic trajectories of the flow with a fixed energy level, which is called the cleanness condition for the flow. It is a generalization of the non-degeneracy condition for isolated periodic trajectories and is based on the notion of clean intersection of submanifolds  introduced by R. Bott \cite{Bott}.

Denote $X_E=H^{-1}(E)\subset T^*M$. It is easy to see that $E>1$ is a regular value of $H$, and, therefore, $X_E$ is a smooth submanifold of  $T^*M$.

\begin{defn}
We say that the flow $\phi$ is clean on $X_E$, if the set 
$\mathcal P=\{(T,x)\in \mathbb R\times X_E : \phi^T(x)=x\}$
is a submanifold of $\mathbb R\times X_E$, and, for any $(T,x)\in \mathcal P$, the following identity holds: 
\[
T_{(T,x)}\mathcal P=\{(\tau,v)\in T_{(T,x)}(\mathbb R\times X_E): d\phi_{(T,x)}(\tau,v)=v\},
\]
where $\phi : \mathbb R \times X_E \to X_E, (t,x)\mapsto \phi^t(x)$.
\end{defn}

If the set of periods of the flow is discrete, cleanness of the flow is equivalent to the following condition: for every period $T$, the fixed point set 
\[
\mathcal P_T = \{x\in X_E : \phi^T(x)=x\}
\] 
of $\phi^T$ is a submanifold of $X_E$, and at each $x\in \mathcal P_T$ the tangent space of $\mathcal P_T$ coincides with the set of fixed vectors of the tangent map $d(\phi^T)_x$. 

Now suppose that the magnetic field form $F$ satisfies the condition \eqref{e:prequantum}. Recall the notion of action for closed curves in  $T^*M$. Let $\eta$ be the canonical 1-form on $T^*M$:
\[
\eta=\sum_{j=1}^np_jdx^j, \quad d\eta=\Omega_0.
\]
In the case when $F$ is exact, i.e. $F=dA$ for some real-valued 1-form $A$, the action  $S_\gamma$ of a closed curve $\gamma$ in $T^*M$ is defined by
\[
S_\gamma=\int_\gamma \eta+\pi^*_M A.
\]
In the case of an arbitrary form $F$, the action of $\gamma$ is defined modulo multiplies of $2\pi$. Denote by $h_A(\gamma)\in S^1=\RR/2\pi \ZZ$ the holonomy of the projection $\pi_M\circ \gamma$ of the curve $\gamma$ to $M$ with respect to the connection $\nabla^L$ on $L$. Then the action $S_\gamma$ of $\gamma$ is given by
\[
S_\gamma=\int_\gamma \eta+ h_A(\gamma).
\]

Suppose that a curve $\gamma$ is given by a periodic solution of the system \eqref{e:1.6} on $X_E$ with period $T$. Then, after a time change, it is given by a periodic solution of the system \eqref{e:1.5} with period $L$, which coincides with the length of the closed curve $\pi_M\circ \gamma$. It  is easy to see that
\[
L=\frac{\sqrt{E^2-1}}{E}|T|.
\]
Since the Hamiltonian $\mathcal H$ given by \eqref{e:1.2a} is homogeneous of degree 2 in momenta, the relation $\langle \eta, \Xi_{\mathcal H}\rangle =2\mathcal H$ holds  for the corresponding Hamiltonian vector field $\Xi_{\mathcal H}$ (with respect to $\Omega$). This implies that
\[
\int_\gamma\eta=\pm L\sqrt{E^2-1},
\]
where the sign $\pm$ depends on the orientation of $\gamma$.

Thus, if the form $F$ is exact, $F=dA$ for some real-valued form $A$, the action of a periodic trajectory $\gamma$ of the flow $\phi$ on $X_E$ of length $L$ is given by
\begin{equation}\label{e:Sgamma}
S_\gamma=\pm L\sqrt{E^2-1}+\int_{\pi_M\circ \gamma} A,
\end{equation}
and, in the case of an arbitrary form $F$, it is given by
\begin{equation}\label{e:Sgamma-gen}
S_\gamma=\pm L\sqrt{E^2-1}+ h_A(\gamma).
\end{equation}

\subsection{The trace formula}
The Guillemin-Uribe trace formula describes the asymptotic expansion, as $N\to\infty$, of the sequence $Y_N$ given by \eqref{e:Yp} with some $E>1$ and $\varphi\in \mathcal S(\RR)$ under the assumption on the associated magnetic geodesic flow to be clean on $X_E$ (cf.  \cite[Corollary 7.2]{Gu-Uribe89} as well as \cite[Theorem 2.1]{Brum-Uribe91}).

\begin{thm}\label{t:varphi}
Assume that the flow $\phi$ is clean on $X_E$ for some $E>1$. Then, for any $\varphi\in \mathcal S(\RR)$, the sequence $Y_N(\varphi)$ admits an asymptotic expansion (the trace formula)
\begin{equation}
\label{traceformula}
Y_N(\varphi)=\sum_{j=0}^{\infty}\varphi(\lambda_{N,j}-EN)\sim \sum_{j=0}^\infty c_j(N,\varphi)N^{d-j},\quad N\to \infty,
\end{equation}
where the coefficients $c_j(N,\varphi)$ are bounded in $N$. Moreover, the following statements about the leading coefficient $c_0$ in the expansion \eqref{traceformula} and the degree $d=d(\varphi)$ hold.
 
(1) If $0$ is a unique period in ${\rm supp}(\hat\varphi)$, where $\hat\varphi$ is the Fourier transform of $\varphi$, then $d=n-1$ and
\begin{equation} \label{e:c0-gen0}
c_0(N,\varphi)=(2\pi)^{-n}\hat{\varphi}(0){\rm Vol}(X_E).
\end{equation}

(2) More generally, assume that there is a unique period $T$ of the flow $\phi$  in the support of $\hat\varphi$. Let $Y_1,\ldots,Y_r$ be the connected components of $\mathcal P_T$ of maximal dimension (denote it by $k$). Then $d=(k-1)/2$, and, for each $j=1,\ldots, r$, there is a density $\alpha_j$ on $Y_j$ defined in terms of the classical dynamics such that
\begin{equation} \label{e:c0-gen1}
c_0(N,\varphi)= \hat\varphi(T)\sum_{j=1}^r e^{\pi im_j/4}e^{-iNS_j}\int_{Y_j}\alpha_j,
\end{equation}
where $m_j = m_{Y_j}$ is the (common) Maslov index of the trajectories in $Y_j$ and $S_j$ is their (common) action.
\end{thm}

If there is a unique periodic trajectory $\gamma\subset X_E$, whose period $T_\gamma$ is in the support of $\hat\varphi$, then $d = 0$ and the assumption on the flow $\phi$ to be clean on $X_E$ implies invertibility of the map $I-P_\gamma$, where $P_\gamma$ denotes the Poincar\'e map of $\gamma$. The density $\alpha$ on $\mathcal P_T$ can be computed explicitly, that results in following formula:
\begin{equation} \label{e:c0-nondegenerate}
c_0(N,\varphi)=\frac{T^\#_\gamma e^{\pi im_\gamma/4}}{2\pi |I-P_\gamma|^{1/2}}e^{-iNS_\gamma}\hat\varphi(T_\gamma)
\end{equation}
where $T^\#_\gamma$ and $m_\gamma$ are  the primitive period and the Maslov index of $\gamma$, respectively.

\subsection{Maslov indices}
Recall the definition of Maslov indices $m_Y$ in the formula \eqref{e:c0-gen1} of Theorem \ref{t:varphi}.

Let $V$ be a symplectic vector space of dimension $2n$. The Lagrangian Grassmannian is the set  $\Lambda=\Lambda(V)$ of Lagrangian subspaces of $V$. It is a regular algebraic subvariety of dimension $n(n+1)/2$ in the Grassmann-variety of all $n$-dimensional linear subspaces of $V$. For any $\lambda\in\Lambda(V)$ and $k=0,1,\ldots, n$, consider the set 
\[
\Lambda^k(\lambda)=\{\mu\in\Lambda(V) :\dim \lambda\cap \mu=k\}.
\] 
$\Lambda^0(\lambda)$ is an affine space associated with the vector space $S^2V$ --- the symmetric tensor power of $V$, $\Sigma(\lambda):=\Lambda\setminus \Lambda^0(\lambda)$  is an algebraic subvariety of codimension 1 in $\Lambda(V)$, whose regular part is $\Lambda^1(\lambda)$ and singular part is $\cup_{k^\prime\geq 2}\Lambda^{k^\prime}(\lambda)$ of codimension $3$ in $\Lambda$. Moreover, $\Lambda^1(\lambda)$ is oriented and connected. The variety $\Sigma(\lambda)$ defines an oriented cycle of codimension 1 in $\Lambda$, independent of the choice of $\lambda\in\Lambda$ and Poincar\'e dual to the generator of $H^1(\Lambda,\mathbb Z)\cong \mathbb Z$. If $\omega: S^1\to \Lambda$ is a continuous loop in $\Lambda$, then its Maslov index $[\omega]$ is defined to be the intersection number of $\omega$ with $\Sigma(\lambda)$. The class $[\omega]$ is independent of the choice of $\lambda$.

The definition of the Maslov index can be extended to non-closed paths in the following way. Let $\lambda\in\Lambda$ and let $\omega : [0,T]\to \Lambda$ be a curve $\omega : [0,T]\to \Lambda$ such that $\omega(0)$ and $\omega(T)$ are complimentary to $\lambda$ and $\omega(0), \omega(T) \in \Lambda^0(\lambda)$. Define the intersection number $[\omega:\lambda]$ of $\omega$ with $\lambda$ in the following way. Let $\omega^\#$ be the loop in $\Lambda$ consisting of $\omega$ followed by any path $\omega^\prime$ in $\Lambda^0(\lambda)$ from $\omega(T)$ to $\omega(0)$. Then
\begin{equation}\label{e:index}
[\omega:\lambda]=[\omega^\#].
\end{equation}
Since $\Lambda^0(\lambda)$ is simply connected, this definition is independent of the choice of $\omega^\prime$. 

Now let us go back to the situation considered in the previous subsection.

Let $(M,g)$ be a compact Riemannian manifold of dimension $n$ and let $(L,h^L)$ be a Hermitian line bundle on $M$ with a Hermitian connection $\nabla^L$. Let $\Pi : S\to M$ be the orthonormal frame bundle of $L$: $S=\{x\in L : |x|_{h^L}=1\}$. It is a principal $S^1=\RR/2\pi\ZZ$-bundle on $M$. We will denote by $e^{i\theta}\cdot x$ the action of $\theta\in S^1$ on $x\in S$ given by the complex multiplication in the fibers of $L$. Let $\partial/\partial\theta$ denote the infinitesimal generator of the $S^1$-action on $S$.

The connection $\nabla^L$ induces a connection on the principal bundle $\Pi : S\to M$. There exists a unique $S^1$-invariant Riemannian metric on $S$ such that $\Pi$ is a Riemannian submersion: the length of $\partial/\partial\theta$ equals 1, the horizontal subspaces of the connection on $S$ are orthogonal to the fibers of $\Pi$ and are mapped isometrically by the map $\Pi$ to the corresponding tangent spaces of the manifold $M$. The induced map $T^*\Pi : T^*S\to T^*M$ is well-defined. The magnetic geodesic flow $\phi$ on $T^*M$ coincides with the Hamiltonian reduction of the Riemannian geodesic flow $f$ on $T^*S$.

For $E>1$, put 
\[
W=\left\{\nu \in T^*S : \|\nu\|=E, \left\langle \nu, \frac{\partial}{\partial\theta}\right\rangle=1\right\}.
\] 
Denote by $\pi : Z \to T^*M$ the restriction of $\Pi$ to 
\[
Z=\left\{\nu \in T^*S : \left\langle \nu, \frac{\partial}{\partial\theta}\right\rangle=1\right\}.
\] 
It is easy to see that $\pi : Z\to X$ and $\pi_E : W\to X_E$ are principal $S^1$-bundles. 

Let $T$ be a period of the flow $\phi$ on $X_E$, let $Y$ be a connected component of $\mathcal P_T= \{x\in X_E : \phi^T(x)=x\}$ of dimension $k$ and $z\in Y$. One can show (see \cite[ Lemma 4.1]{Gu-Uribe89}) that there is $\theta \in \RR$ such that, for any $\nu\in \pi_E^{-1}(z)$, we have 
\[
e^{i\theta}\cdot f^T(\nu)=\nu.
\] 
Choose some $\nu\in \pi_E^{-1}(z)$ and put $\mu=f^T(\nu)$. Thus, $\mu\in \pi^{-1}(z)$ and $e^{i\theta}\cdot \mu=\nu$. Choose local coordinate system with coordinates $x\in \RR^n$ in a neighborhood $U$ of $\pi(z)\in M$ and denote by $\xi\in \RR^n$ the dual coordinates in the cotangent bundle $T^*M$. We can assume that a trivialization of the bundle $\pi_E$ is defined over $U$. Thus, $\pi^{-1}(U)\cong U\times S^1$ is a coordinate neighborhood of $z$ with the coordinates  $(x,\theta)\in U\times S^1$ and the dual coordinates $(\xi, \tau)\in \RR^{n+1}$ in the cotangent bundle $T^*S$. The choice of coordinates determines horizontal Lagrangian subspace $H_{\nu}$ of $T_\nu(T^*S)$ defined by the equations $\delta \xi=0,\delta \tau=0$. Consider the curve $\omega$ in the Lagrangian Grassmanian $\Lambda(T_\nu(T^*M))$ defined by 
\begin{equation}\label{e:vert-f}
\omega(t)=df^{-t}(V_{f^t(\nu)}), \quad 0\leq t\leq T,
\end{equation}
where $V_{f^t(\nu)}\subset T_{f^t(\nu)}(T^*S)$ is the vertical subspace. Denote by $\kappa$ the intersection number of the curve $\omega$ with $H_\nu\in \Lambda(T_\nu(T^*S))$ defined by \eqref{e:index}:
\begin{equation}\label{e:kappa}
\kappa=[\omega:H_\nu].
\end{equation}
Let $\chi(x,\eta)$ be the generating function of the canonical transformation $f^T$ of $T^*S$. It means that 
$\det(\partial^2\chi/\partial x\partial\eta)\neq 0$ and in some neighborhood of $(\nu,\mu)\in T^*S\times T^*S$ the graph $C$ of $f^T$ has the form
\[
C=\{(x,\chi^\prime_x,\chi^\prime_\eta,\eta)\in T^*S\times T^*S\},
\]
i.e., $f_T(x,\chi^\prime_x)=(\chi^\prime_\eta,\eta)$. Consider the matrix
\begin{equation}\label{e:R}
R=\begin{pmatrix}
d^2_{xx}\chi & d^2_{x\eta}\chi & -1 \\
d^2_{x\eta}\chi & d^2_{\eta\eta }\chi & 0 \\
-1 & 0 & 0
\end{pmatrix}.
\end{equation}

The integer $m_Y$ associated to the connected component $Y$ is given by
\begin{equation}\label{e:m}
m_Y = {\rm sgn}\,R+2\kappa.
\end{equation}
This formula agrees with a more general definition given in  \cite[Proposition 3.7]{Gu-Uribe89}.

\subsection{On the proof of the trace formula}
In this section, we will briefly describe the main ideas of the paper of V. Guillemin and A. Uribe, which is based on a further elaboration of the paper by Y. Colin de Verdi\`ere   \cite{ColinDeV79}.

Let $E>1$ and $\varphi\in \mathcal S(\RR)$. Consider the distribution $Y\in \mathcal D^\prime(\mathbb R)$ given by
\begin{equation}\label{e:YE}
Y(s)=\sum_{N=1}^{\infty}Y_N(\varphi)e^{iNs}=\sum_{N=1}^{\infty}\sum_{j=0}^{\infty}\varphi(\lambda_{N,j}-EN)e^{iNs}.
\end{equation}
It is a $2\pi$-periodic distribution 
\[
Y(s+2\pi)=Y(s)
\] 
and belongs to the generalized Hardy space, i.e., its Fourier series involves only positive frequencies. 

The key idea is to interpret $N$ as an eigenvalue of the operator $i^{-1}\partial_\theta$ on the unit circle $S^1=\RR/2\pi \ZZ$ with the coordinate $\theta$. In local coordinates, this means that it is supposed to consider instead of the operator  \eqref{e:DLp} a so called horizontal Laplacian $\Delta_{\rm h}$ on $M\times S^1$ given by 
\begin{multline*}
\Delta_{\rm h}=-\frac{1}{\sqrt{|g(x)|}}\sum_{1\leq j,\ell\leq n}\left(\frac{\partial}{\partial x^j}-A_j(x)\frac{\partial}{\partial \theta}\right)\times \\ \times \left[\sqrt{|g(x)|}
g^{j\ell}(x) \left(\frac{\partial}{\partial
x^\ell}-A_\ell(x)\frac{\partial}{\partial \theta}\right)\right].
\end{multline*}

In the global setting we consider the principal $S^1$-bundle associated with $L$: 
\[
S=\{x\in L : |x|_{h^L}=1\}.
\] 
The connection on $L$ gives rise to a connection on $S$, i.e. an $S^1$-invariant distribution, transversal to the fibers of the fibration $\Pi : S\to M$ (the horizontal distribution of the connection). The horizontal Laplacian $\Delta_{\rm h}$ is naturally determined by the horizontal distribution and the Riemannian metric on $M$.

The eigenvalues of $\Delta_{\rm h}$ are described as follows (cf., for instance, \cite[Lemma 5.4]{Gu-Uribe86}). For any $k\in \ZZ$, consider the space $E_k$ of smooth functions on $S$ such that $f(e^{i\theta}\cdot x)=e^{-ik\theta}f(x)$ for any $x\in S$ and $\theta\in S^1$. Functions from $E_k$ can be identified with sections of the vector bundle $S\times_{S^1}\CC\cong L^k$ on $M$ associated with the principal bundle $\Pi : S\to M$ and the unitary representation $S^1\ni \theta \mapsto e^{ik\theta}\in \CC$ of the group $S^1$. Under the isomorphism $E_k\cong C^\infty(M,L^k)$, the restriction of $\Delta_{\rm h}$ to the subspace $E_k$ corresponds to the magnetic Laplacian $\Delta^{L^k}$. Therefore, the eigenvalues of $\Delta_{\rm h}$ are of the form $\{\nu_{k,j}, j\in \NN, k\in \ZZ\}$. The operator $\Delta_{\rm h}$ is not an elliptic operator, therefore, it is more convenient to replace it with an elliptic operator.

Let $\partial/\partial\theta$ be the infinitesimal generator of the $S^1$-action on $S$. Fix an $S^1$-invariant Riemannian metric on $S$ such that $\Pi : S\to M$ is a Riemannian submersion and, moreover, the horizontal subspaces of the connection on $S$ are orthogonal to the fibers of the fibration $\Pi$. Let $\Delta$ be the associated Laplace-Beltrami operator on $S$. It has the form
\[
\Delta=\Delta_{\rm h}-\frac{\partial^2}{\partial \theta^2}.
\]
Consider the operator $P=\Delta^{1/2}$, which is a first order elliptic pseudodifferential operator on $S$. Its eigenvalues have the form $\{\lambda_{k,j}, j\in \NN, k\in \ZZ\}$ (cf. \eqref{e:def-lambda}). Moreover, if $A$ is the first order differential operator 
\[
A=\frac{1}{i}\frac{\partial}{\partial \theta},
\] 
then its eigenvalues are integers, and the eigenspace, corresponding to the eigenvalue $k\in \ZZ$, is $E_{-k}$. The operators $P$ and $A$ commute, and their joint eigenvalues are of the form  $(\lambda_{k,j},-k)$.

The distribution $Y\in \mathcal D^\prime(\RR)$ can be interpreted as the distributional trace of the operator $\varphi(P-EA)e^{is A}$: for any test function $f\in C^\infty_c(\RR)$,
\[
\langle Y, f\rangle =\operatorname{tr} \int_{-\infty}^{+\infty}\varphi(P-EA)e^{is A}f(s)ds=2\pi\operatorname{tr} \varphi(P-EA)\check{f}(A),
\]
where $\check{f}$ denotes the inverse Fourier transform of $f$. Here the operator $P-EA$ is not necessarily elliptic. Therefore, the operator $\varphi(P-EA)$ is not in general of trace class. But the operators $P-EA$ and $A$ are commuting jointly elliptic operators on $S$, that allows one to prove that the operator $\varphi(P-EA)\check{f}(A)$ is a smoothing operator and, therefore, its trace is well-defined.

In \cite{Gu-Uribe89},  V. Guillemin and A. Uribe made an analysis of the distribution $Y$  in the spirit of the proof of the Duistermaat-Guillemin trace formula  \cite{DG75}. First of all, they gave the following description of the singularities of the distribution $Y$ (see \cite[Theorem 6.9]{Gu-Uribe89}).

\begin{thm}\label{t:singY}
If $Y$ is not smooth in $s\in \mathbb R$, then there exists a closed curve $\gamma$ in $M$, satisfying the equation \eqref{e:magnetic} and the condition
$\|\dot{\gamma}\|=c:=\sqrt{E^2-1}$, such that 
\[
s=S_\gamma.
\]
\end{thm}

\begin{defn}
Let $Y$ be a distribution on the real line and let $s\in \mathbb R$ be a singularity of $Y$. We will say that $s$ is a classical singularity of degree $d$ if
\begin{itemize}
\item  it is an isolated singular point;
\item for any $\rho\in C^\infty_c(\mathbb R)$ such that $\rho\equiv 1$ in a neighborhood of $s$ and there is no other singularity of $Y$ in the support of $\rho$, we have 
\[
\widehat{\rho Y}(\xi)\sim e^{-is\xi}\sum_{j=0}^\infty c^{\pm}_j\xi^{d-j}, \quad \xi\to \pm \infty.
\]
\end{itemize}

If, in addition, $c_j^-=0$ for all $j\geq 0$, we will say that such a classical singularity is positive.
\end{defn}

\begin{thm}{\cite[Theorems 6.9, 4.7]{Gu-Uribe89}}\label{t:1.3}
If the flow $\phi$ is clean on $X_E$, then the singularities of $Y$ are classical and positive.
\end{thm}

The character of singularities of $Y$ reflects in the asymptotic behavior of its Fourier coefficients $Y_N(\varphi)$. This is stated more precisely in the following lemma, which allows one to derive from Theorem  \ref{t:1.3} the asymptotic expansion of the sequence $Y_N(\varphi)$ as $N\to\infty$.

Consider a $2\pi$-periodic distribution 
\[
Y(s)=\sum_{N=0}^\infty a_Ne^{iNs}.
\]
Suppose that $Y$ has finitely many singular points $\omega_1,\ldots,\omega_K$ on the circle $S^1$, which are positive classical singularities. For every $j=1,\ldots,K$, choose $u_j$ such that $-iu_j=\log \omega_j$ and denote by $d_j$ the degree of the corresponding singularity. By assumption, if $\rho_j$ is a cut-off function near $u_j$, we have an asymptotic expansion of the form
\[
\widehat{\rho_jY}(\xi)\sim \sum_{l=0}^\infty \alpha_j(l)\xi^{d_j-l}.
\]

\begin{lem}{\cite[Lemma 7.1]{Gu-Uribe89}}
As $N\to \infty$, the Fourier coefficients of $Y$ admit the asymptotic expansion 
\[
a_N\sim \sum_{l=0}^\infty\sum_{j=1}^K\alpha_j(l)\omega_j^{-N}N^{d_j-l}.
\]
\end{lem}

\subsection{Periodic flows}
The case of periodic magnetic flows was considered in \cite[Section 4]{Brum-Uribe91}. Assume that, for some $E>1$, the manifold $X_E$ is connected, and each point in $X_E$ is periodic with minimal period $T>0$. For any $x\in X_E$, denote by $\gamma_x$ the trajectory $\phi^t(x), t\in [0,T]$. One can show that the action $x\mapsto e^{iS_{\gamma_x}}$ is constant on $X_E$. Assume that the trajectories $\gamma_x$ satisfy the Bohr-Sommerfeld quantization condition:
\[
e^{iS_{\gamma_x}}=1.
\]
It follows from Theorem \ref{t:singY} that the singularities of the distribution $Y\in \mathcal D^\prime(\mathbb R)$ given by \eqref{e:YE} have the form $s=2\pi k, k\in \ZZ$.

\begin{thm}{\cite[Theorem 4.8]{Brum-Uribe91}} \label{t:periodic}
Suppose that an energy level $E$ satisfies the above conditions. 
There are constants $C_1, C_2>0$ and an integer $\sigma$ such that, if $N$ is sufficiently large, the following inequality holds:
\[
\#\left\{j: \left|\lambda_{N,j}-EN-\frac{\pi\sigma}{2T}\right|\leq C_1\right\}\leq \#\left\{j: \left|\lambda_{N,j}-EN-\frac{\pi\sigma}{2T}\right|\leq C_2N^{-1}\right\}.
\]
Denote by $J(N)$ the set of indices on the right hand side of this inequality: 
\[
J(N)=\left\{j: \left|\lambda_{N,j}-EN-\frac{\pi\sigma}{2T}\right|\leq C_2N^{-1}\right\}.
\]
Then, for large $N$, the cardinality of $J(N)$ is a polynomial of $N$ of degree $n-1$, whose leading coefficient equals the Liouville measure of $X_E$ divided by $(2\pi)^n$.
\end{thm}

Moreover, one can claim that there is an integer $N_0$ such that, for sufficiently large $N$, the cardinality of $J(N)$ is equal to $P(N+N_0)$, where $P$ is the Riemann-Roch polynomial of the symplectic manifold  $X_E/S^1_T$, where the group $S^1_T=\RR/T\ZZ$ acts on $X_E$ via the flow $\phi^t$.

Finally, by \cite[Proposition 4.10]{Brum-Uribe91}, for any function $\varphi\in C^\infty_c(\mathbb R)$, there is an asymptotic expansion
\[
\sum_{j\in J(N)}\varphi\left(N\left(\lambda_{N,j}-EN-\frac{\pi\sigma}{2T}\right)\right)\sim \sum_{k=0}^{\infty}N^{n-1-k}\beta_k(\varphi),
\]
where
\[
\beta_0(\varphi)=\int_{X_E}\varphi\circ \rho\, d{\rm vol}.
\]
Here ${\rm vol}$ denotes the Liouville measure on $X_E$ and $\rho$ is some function on $X_E$.

\section{Examples} \label{s:examples}
\subsection{Constant magnetic field on the two-torus}\label{s:torus}
Consider the two-torus
\[
\mathbb T^2=\mathbb R^2/\mathbb Z^2,
\]
equipped with the standard flat Riemannian metric
\[
g=dx^2+dy^2.
\]
Let $L$ be the line bundle over $\mathbb T^2$, whose sections are identified with functions $u\in C^\infty(\RR^2)$, satisfying the conditions
\begin{equation}\label{e:torus-sections}
u(x+1,y)=e^{iBy}u(x,y),\quad u(x,y+1)=u(x,y).
\end{equation}
Such functions are examples of Bloch functions in the magnetic field \cite{Novikov81,Novikov83}.

The periodicity condition for the function $e^{iBy}$ gives a quantization condition
\begin{equation}\label{e:torus-quant}
B=2\pi n, \quad n\in \ZZ.
\end{equation}
Put
\[
B=2\pi.
\]
Then the Hermitian connection on $L$ can be defined as
\[
\nabla^L=d-iA = d - 2\pi i xdy.
\]
The magnetic field form $F$ is given by
\[
F=2\pi dx\wedge dy.
\]

\begin{thm}\label{t:5}
For any $\varphi\in \mathcal S(\RR)$ and for any $E>1$,  one has an asymptotic expansion
\begin{equation}\label{e:2.1}
Y_N(\varphi) \sim \sum_{j=0}^\infty c_j(N,\varphi)N^{1-j},\quad N\to \infty.
\end{equation}
The coefficients $c_j$ can be computed explicitly. For the first two of them, we get
\begin{align}
c_0(N,\varphi)=&\frac{E}{2\pi}\sum_{k\in \ZZ}\hat\varphi (kE)\exp(ik\pi)\exp\left(-ik \frac{E^2-1}{2}N\right),\label{e:c0-torus}\\
c_1(N,\varphi)
=&-\sum_{k\in \ZZ}\left(\frac{i}{2\pi E^2}\hat{\varphi}^{\prime}\left(kE\right)+\frac{i k}{4\pi E}\hat{\varphi}^{\prime\prime}\left(k E\right)\right)\exp\left(-ik \frac{E^2-1}{2}N\right).\label{e:c1-torus}
\end{align}
\end{thm}

\begin{proof}
We recall the computation of its eigenvalues and eigenfunctions of the operator $\Delta^{L^N}$. Let $u\in C^\infty(\mathbb T^2, L)$ be  an eigenfunction of $\Delta^{L^N}$ with the corresponding eigenvalue $\nu$. So $u$ is a smooth function on $\RR^2$, satisfying \eqref{e:torus-sections}, such that
\[
\Delta^{L^N}u(x,y)=\nu u(x,y), \quad (x,y)\in \mathbb R^2.
\]
The operator $\Delta^{L^N}$ has the form
\[
\Delta^{L^N}=-\frac{\partial^2}{\partial x^2}-\left(\frac{\partial}{\partial y}-2\pi Ni x\right)^2.
\]
Let us expand the function $u$ in the Fourier series in $y$:
\[
u(x,y)=\sum_{k\in\ZZ} u_k(x)e^{2\pi iky}, \quad (x,y)\in \mathbb R^2.
\]
The functions $u_k$ satisfy the ordinary differential equations
\[
\left[-\frac{d^2}{d x^2}+4\pi^2N^2\left(x-\frac{k}{N}\right)^2\right]u_k(x)=\lambda u_k(x).
\]
The first condition in \eqref{e:torus-sections} implies that
\begin{equation} \label{e:torus-sections1}
u_k(x+1)=u_{k-N}(x), \quad x\in \mathbb R, \quad k\in \ZZ.
\end{equation}
Using a well-known computation of the spectrum of the quantum Hamiltonian of the harmonic oscillator, we infer that $\nu$ is of the form (Landau levels)
\[
\nu_{N,j}=2\pi N(2j+1),\quad j=0,1,2,\ldots,
\]
and the corresponding eigenfunction $u_k$ is given by
\[
u_{k,j}(x)=a_kh_j\left(2\pi N\left(x-\frac{k}{N}\right)\right),
\]
where $h_j$ are the Hermite functions.
Taking into account \eqref{e:torus-sections1}, it is easy to see that an arbitrary eigenfunction of $\Delta^{L^N}$ with eigenvalue $\nu=\nu_{N,j}$ has the form
\[
u(x,y)=\sum_{k=0}^{N-1} a_k \sum_{\ell\in \ZZ} h_j\left(2\pi (Nx-k+N\ell)\right)e^{2\pi i(k-N\ell)y}
\]
with arbitrary $a_0, a_1, \ldots, a_{N-1}$. In particular, for the multiplicity $m_{N,j}$ of the eigenvalue $\nu_{N,j}$, we have
\[
m_{N,j}=N.
\]

According to the above description of the spectrum of  $\Delta^{L^N}$, we get the following expression for the function $Y_N(\varphi)$:
\begin{equation} \label{e:torus-Yp}
Y_N(\varphi)=\sum_{j=0}^{\infty}N\varphi(\sqrt{N^2+2\pi N(2j+1)}-EN).
\end{equation}
Let us make the change of the summation index in the right hand side of \eqref{e:torus-Yp}:
\[
j=\frac{E^2-1}{4\pi}N-\left\{\frac{E^2-1}{4\pi}N\right\}+n, \quad n\in \ZZ,
\]
where $\{\cdot \}$ denotes the fractional part of a number. 
Then this formula reads as follows:
\begin{multline*}
Y_N(\varphi)\\ =\sum_{n=-\left[\frac{E^2-1}{4\pi}N\right]}^{\infty}N\varphi\left(\sqrt{E^2N^2+2\pi N(2n+1)-4\pi N\left\{\frac{E^2-1}{4\pi}N\right\}}-EN\right).
\end{multline*}
Since $\varphi\in \mathcal S(\RR)$, it is clear that
\begin{multline} \label{e:torus-Ypn}
Y_N(\varphi)\\ =\sum_{n\in \ZZ}N\varphi\left(\sqrt{E^2N^2+2\pi N(2n+1)-4\pi N\left\{\frac{E^2-1}{4\pi}N\right\}}-EN\right)+O(N^{-\infty}).
\end{multline}

Recall the Taylor expansion for $f(x)=\sqrt{1+x}-1$ at $x=0$:
\begin{equation} \label{e:Taylor-f}
f(x)=\sqrt{1+x}-1=\frac 12x-\frac 18 x^2+\sum_{k=3}^\infty c_kx^k, \quad |x|<1,
\end{equation}
where
\[
c_k=(-1)^{k-1}\frac{(2k-3)!}{2^{2k-2}k!(k-2)!}, \quad k\geq 2.
\]
For each $n$, we have the following asymptotic expansion for the argument of $\varphi$ in \eqref{e:torus-Ypn} as $N\to \infty$:
\begin{multline*}
\sqrt{E^2N^2+2\pi N(2n+1)-4\pi N\left\{\frac{E^2-1}{4\pi}N\right\}}-EN \\
\begin{aligned}
=&\frac{1}{E} \left(\pi (2n+1)-2\pi\left\{\frac{E^2-1}{4\pi}N\right\}\right)\\ & -\frac{1}{8E^3N} \left(2\pi (2n+1)-4\pi\left\{\frac{E^2-1}{4\pi}N\right\}\right)^2\\ &+\sum_{k=3}^\infty \frac{c_k}{E^{2k-1}N^{k-1}} \left(2\pi (2n+1)-4\pi\left\{\frac{E^2-1}{4\pi}N\right\}\right)^k.
\end{aligned}
\end{multline*}
Using this expansion and the Taylor expansion of $\varphi$:
\begin{multline}\label{e:Taylor-phi}
\varphi(a_0+a_1N^{-1}+a_2N^{-2}+\ldots)\\=\varphi(a_0)+\varphi^\prime(a_0)a_1N^{-1}+\left(\varphi^\prime(a_0)a_2+\frac 12\varphi^{\prime\prime}(a_0)a_1^2\right)N^{-2}+\ldots,
\end{multline}
we obtain the asymptotic expansion \eqref{e:2.1}.

Let us compute the first two terms of \eqref{e:2.1}. 
For the coefficient at $N$ we get
\begin{equation}\label{e:c0a-torus}
c_0(N,\varphi) =\sum_{n\in \ZZ} \varphi\left(\frac{\pi (2n+1)}{E}-\frac{2\pi }{E}\left\{\frac{E^2-1}{4\pi}N\right\}\right).
\end{equation}
Recall the Poisson summation formula 
\begin{equation}\label{e:Poisson}
\sum_{n\in \ZZ} f(nP+t)=\sum_{k\in \ZZ}\frac{1}{P}\hat{f}\left(\frac{2\pi k}{P}\right)e^{2\pi i\frac{k}{P}t},
\end{equation}
where the function $f$ belongs to $\mathcal S(\RR)$ and its Fourier transform $\hat f$ is given by
\[
\hat{f}(k)=\int f(x)e^{- ikx}dx,\quad k\in \RR.
\]
Applying \eqref{e:Poisson} in the right-hand side of \eqref{e:c0a-torus}, we immediately get \eqref{e:c0-torus}:
\begin{align*}
c_0(N,\varphi) = & \sum_{k\in \ZZ} \frac{E}{2\pi}\hat\varphi (kE)e^{ik\pi}e^{-2\pi ik \left\{\frac{E^2-1}{4\pi}N\right\}}\\
= & \frac{E}{2\pi} \sum_{k\in \ZZ} \hat\varphi (kE)e^{ik\pi}e^{-ik \left(\frac{E^2-1}{2}N\right)}.
\end{align*}

For the next coefficient, we have
\begin{multline*}
c_1(N,\varphi)
=-\sum_{n\in \ZZ} \frac {1}{2E^3} \varphi^\prime \left(\frac{\pi (2n+1)}{E}-\frac{2\pi }{E}\left\{\frac{E^2-1}{4\pi}N\right\}\right)\times \\ \times \left(\frac{\pi (2n+1)}{E}-\frac{2\pi }{E}\left\{\frac{E^2-1}{4\pi}N\right\}\right)^2.
\end{multline*}
By the Poisson summation formula \eqref{e:Poisson} with $f=\varphi(x)x^2$, we get \eqref{e:c1-torus}. Theorem~\ref{t:5} is proved.
\end{proof}

We give a geometric interpretation of the formula \eqref{e:c0-torus}, comparing it with the general formulas \eqref{e:c0-gen0} and \eqref{e:c0-gen1}.

The Hamiltonian $H$ given by \eqref{e:H} has the form
\[
H(x,y,p_x,p_y)=\left(p_x^2+p_y^2+1\right)^{1/2}.
\]
The Liouville form on $X_E = H^{-1}(E) \subset T^*M$ is written as
\[
{\rm vol}=\frac{dx\,dy\,dp_x\,dp_y}{dH},
\] 
and, by simple computations, we get
\[
{\rm Vol}(X_E)= 2\pi E{\rm Vol}(M)=2\pi E.
\]
This shows that \eqref{e:c0-torus} for $k=0$ agrees with \eqref{e:c0-gen0}.

For an arbitrary (not necessarily satisfying the quantization condition) value of the constant magnetic field $B$, trajectories of a particle are well-known  --- they are ``Larmor orbits'', which are taken to the circle of radius 
\[
r = \frac{\sqrt{2\cE}}{B} = \frac{\sqrt{E^2-1}}{B},
\] 
clockwise oriented, by the lift to the universal covering $f: \RR^2 \to \mathbb T^2$.  The period of these trajectories (as trajectories of the Hamiltonian system defined by $H$) equals $T=2\pi E/B$. Since the magnetic geodesic flow is periodic, it is clean on $X_E$ and $d=n-1=2$.

Assume now that $B$ satisfies the condition \eqref{e:torus-quant}. Compute $h_A(\gamma)$: by Stokes formula, we get
\[
h_A(\gamma)=\int_{f^{-1}(\gamma)} Bxdy = -\int_D B dx \wedge dy,
\]
where the disc $D$ is bounded by the trajectory, and the sign minus appears  because the orientations of the trajectory and the boundary of the disc are different. Finally, we obtain
\[
h_A(\gamma)=- B \int_D dx\wedge dy = -B \pi \frac{E^2-1}{B^2} = - \frac{\pi}{B}(E^2-1).
\]
When $B=2\pi$, by \eqref{e:Sgamma-gen}, we get 
\[
S_\gamma=h_A(\gamma)+L\sqrt{E^2-1}=\frac{1}{2}(E^2-1).
\]
This shows that \eqref{e:c0-torus} for $k\neq 0$ agrees with \eqref{e:c0-gen1}. In particular we conclude that $m=4$. Since the flow is periodic, $f^T={\rm id}$ and, therefore, ${\rm sgn}\,R=0$. On the other hand, we get $\kappa=2$.

Let us illustrate the statement of Theorem~\ref{t:periodic} on this example. It is easy to see that, if $S_\gamma=(E^2-1)/2\in 2\pi\ZZ$ and the eigenvalue $\lambda_{N,j}$ satisfies the condition
\[
\left|\lambda_{N,j}-EN-\frac{\pi}{E}\right|<\frac{2\pi}{E},
\]
then $n=0$. Therefore, $\lambda_{N,j}=\sqrt{E^2N^2+2\pi N}$ and the following relation holds
\[
\left|\lambda_{N,j}-EN-\frac{\pi}{E}\right|= O(N^{-1}).
\]
This statement agrees with Theorem~\ref{t:periodic}, if we put $\sigma=2$ and take into account that $T=E$.

\subsection{The two-sphere}
Let the Riemannian manifold $(M,g)$ be the two-sphere
\[
x^2+y^2+z^2=R^2
\]
with the Riemannian metric induced by the embedding into the Euclidean space $\RR^3$. In the spherical coordinates
\[
x=R\sin\theta \cos\varphi, \quad y=R\sin\theta \sin\varphi, \quad z=R\cos\theta, \quad \theta\in (0,\pi), \varphi\in (0,2\pi),
\]
the Riemannian metric $g$ is given by 
\[
g=R^2(d\theta^2+\sin^2\theta d\varphi^2).
\]
Consider the magnetic field form $F$ of the form
\[
F=B\,d \mathrm{vol}_M=B\sin\theta d\theta\wedge d\varphi.
\]
The quantization condition means that
\begin{equation}\label{e:sphere-quant}
B\in \frac 12\ZZ.
\end{equation}

For $B=n/2, n\in \ZZ,$ the corresponding Hermitian line bundle  $L_n$ can be described as the line bundle associated with the Hopf fibration $S^3\to S^2$ and the character $\chi_n : S^1\to S^1$ given by $ \chi_n(u)=u^n, u\in S^1$. In physics literature, $L_n$ is a well-known Wu-Yang magnetic monopole \cite{Wu-Yang}, which provides a natural topological interpretation of Dirac's monopole of magnetic
charge $g=n\hbar/(2e)$.  

\begin{thm}\label{t:6}
Let $B=1/2$. For any $\varphi\in \mathcal S(\RR)$ and any $E>1$,  one has an asymptotic expansion
\begin{equation}\label{e:3.1}
Y_N(\varphi) \sim \sum_{j=0}^\infty c_j(N,\varphi)N^{1-j},\quad N\to \infty.
\end{equation}
The coefficients $c_j$ can be computed explicitly. For the first two of them, we get
\begin{multline}\label{e:c0-sphere}
c_0(N,\varphi) = \sum_{k\in \ZZ}2ER^2 \hat{\varphi}\left(\frac{2\pi ERk}{\sqrt{E^2-1+\frac {1}{4R^2}}}\right)\\ \times \exp(\pi ik(N+1)) \exp\left(-2\pi ikR\sqrt{E^2-1+\frac {1}{4R^2}}N\right),
\end{multline}
\begin{multline*}
c_1(N,\varphi) = \sum_{k\in \ZZ}\Bigg[2iR^2 \hat{\varphi}^\prime\left(\frac{2\pi ERk}{\sqrt{E^2-1+\frac {1}{4R^2}}}\right)\\
-\frac{\pi ER(4R^2-1)k}{2(E^2-1+\frac {1}{4R^2})^{3/2}}i\hat{\varphi}^{\prime\prime}\left(\frac{2\pi ERk}{\sqrt{E^2-1+\frac {1}{4R^2}}}\right)\\ -\frac{\pi iERk}{2\sqrt{E^2-1+\frac {1}{4R^2}}} \hat{\varphi}\left(\frac{2\pi ERk}{\sqrt{E^2-1+\frac {1}{4R^2}}}\right) \Bigg]\\ \times \exp(\pi ik(N+1)) \exp\left(-2\pi ikR\sqrt{E^2-1+\frac {1}{4R^2}}N\right).
\end{multline*}
In particular, when $R=1/2$, we get
\begin{equation*}
c_0(N,\varphi) = \sum_{k\in \ZZ}\frac{E}{2} \hat{\varphi}\left(\pi k\right)\exp(\pi ik(N+1)) \exp(-\pi ikEN),
\end{equation*}
\begin{equation*}
c_1(N,\varphi) = \sum_{k\in \ZZ}\left[\frac{1}{2}i \hat{\varphi}^\prime\left(\pi k\right) -\frac{\pi ik}{4} \hat{\varphi}\left(\pi k\right) \right]\exp(\pi ik(N+1)) \exp(-\pi ikEN).
\end{equation*}
\end{thm}

\begin{proof}
The spectrum of the magnetic Laplacian $\Delta^{L^N}$ was computed in 
 \cite{Wu-Yang2} (see also \cite{Kuwabara82,FV}). It consists of the eigenvalues  (spherical Landau levels)
\[
\nu_{N,j}=\frac{1}{R^2}\left[j(j+1)+\frac{N}{2}(2j+1)\right], \quad j=0,1,2, \ldots,
\]
with multiplicity
\[
m_{N,j}=N+2j+1.
\]
The corresponding eigenfunctions are known as monopole harmonics.

According to this description of the spectrum, we get the following expression for the function  $Y_N(\varphi)$:
\begin{multline} \label{e:sphere-Yp}
Y_N(\varphi)\\=\sum_{j=0}^\infty(N+2j+1)\varphi\left(\sqrt{N^2+\frac{1}{R^2}\left[j(j+1)+\frac{N}{2}(2j+1)\right]}-EN\right).
\end{multline}
Let us put
\[
\alpha=\left(c^2R^2+\frac 14\right)^{1/2}-\frac 12
\]
and make the change of the summation index
in the right hand side of  \eqref{e:sphere-Yp}:
\[
j=\alpha N-\{\alpha N\}+n, \quad n\in \ZZ.
\]
To simplify the notation, put
\[
P=\frac{2\alpha+1}{2ER^2}, \quad t=\frac{2\alpha+1}{2ER^2}\left(\frac 12-\{\alpha N\}\right).
\]
The numbers $\lambda_{N,j}$ defined by \eqref{e:def-lambda} are of the form
\[
\lambda^2_{N,j} = N^2+\nu_{N,j}\\=E^2N^2+\frac{2\alpha+1}{R^2}\left(n+\frac tP\right)N+\frac{1}{R^2}\left[\left(n+\frac tP\right)^2-\frac 14\right],
\]
and, by use of \eqref{e:Taylor-f}, we get the following asymptotic expansion for the argument of $\varphi$ in \eqref{e:sphere-Yp} as $N\to \infty$:
\[
\lambda_{N,j}-EN=nP+t +\left(\frac{4R^2-1}{8E^3R^4}\left(n+\frac tP\right)^2-\frac{1}{8ER^2}\right)N^{-1}+ O(N^{-2}).
\]
For multiplicities, we have
\[
m_{N,n}=(2\alpha+1)N+2\left(n+\frac tP\right).
\]
As above, using these asymptotic expansions and \eqref{e:Taylor-phi}, we obtain the asymptotic expansion \eqref{e:3.1}.

The computation of the coefficients $c_0$ and $c_1$ in the expansion \eqref{e:3.1} can be made by analogy with the computation given in Subsection~\ref{s:torus}, using the Poisson summation formula \eqref{e:Poisson}, and we will omit it. Theorem~\ref{t:6} is proved.
\end{proof}

We give a geometric interpretation of the formula \eqref{e:c0-sphere}, comparing it with the general formulas \eqref{e:c0-gen0} and \eqref{e:c0-gen1}.

The Hamiltonian $H$ given by \eqref{e:H} has the form
\[
H=\left(\frac{1}{R^2}p^2_\theta+\frac{1}{R^2\sin^2\theta}p^2_\varphi+1\right)^{1/2},
\]
and the corresponding Hamiltonian system is explicitly and easily solved.

Similarly to the case of the two-torus, we compute the Liouville measure on  $X_E = H^{-1}(E)$:
\[
{\rm Vol}(X_E)= 2\pi E{\rm Vol}(M)=8\pi^2 ER^2.
\]
We see that the formula \eqref{e:c0-sphere} for $k=0$ agrees with \eqref{e:c0-gen0}.

All trajectories will be closed again with the same period 
\[
T = \frac{2\pi ER}{\sqrt{c^2+\frac{B^2}{R^2}}}.
\]
The flow is periodic. Therefore, it is clean on $X_E$ and $d=n-1=2$.

For an arbitrary value of the magnetic field $B$ the projections $\pi_M\circ \gamma$ of periodic trajectories $\gamma$ to the sphere are the circles (the sections of the sphere by planes in $\RR^3$) of Euclidean radius
\[
r =\frac{cR}{\sqrt{c^2+\frac{B^2}{R^2}}},
\]
where, we recall, $c = \sqrt{E^2-1} = \sqrt{2\cE}$, and of length $2\pi r$.

Now we assume that $B$ satisfies the quantization condition  \eqref{e:sphere-quant}. Let us compute $h_{A}(\gamma)\in S^1$, the holonomy of $\pi_M\circ \gamma$ with respect to $\nabla^{L}$:
\[
h_{A}(\gamma)=\int_\gamma A,
\]
where $A$ is the connection form of  $\nabla^{L}$ with respect to some trivialization of  $L$ over the upper hemisphere. We have $dA=F= B\,d \mathrm{vol}_M$. Therefore, using Stokes' formula, we get:
\[
h_{A}(\gamma)= B \int_D d\mathrm{vol}_M = B S,
\]
where $D$ is the spherical disc, bounded by the trajectory and entirely belonging to the hemisphere, and $S$ is its area. In the end, we find:
\[
h_{A}(\gamma)=\pi\left(1-\frac{\frac{B}{R}}{\sqrt{c^2+\frac{B^2}{R^2}}}\right)
\]
Finally, by the formula \eqref{e:Sgamma} for $B=1/2$ we compute
\[
S_\gamma=-\pi\left(1-\frac{\frac{1}{2R}}{\sqrt{c^2+\frac{1}{4R^2}}}\right)+\frac{2\pi cR}{\sqrt{c^2+\frac{1}{4R^2}}}\sqrt{E^2-1}=-\pi+2\pi R\sqrt{c^2+\frac{1}{4R^2}}.
\]
Thus, we see that the formula \eqref{e:c0-sphere} for $k\neq0$ agrees with \eqref{e:c0-gen1}.

\subsection{The hyperbolic plane}
Consider the hyperbolic plane $\mathbb H=\{(x,y)\in \RR^2 : y>0\}$ equipped with the Riemannian metric
\[
g=\frac{R^2}{y^2}(dx^2+dy^2).
\]
Assume that the Hermitian line bundle $\tilde L$  is trivial and the connection $\nabla^{\tilde L}$ on $\tilde L$ is given by the connection form
\[
A=\frac{B}{y}dx.
\]
So we have
\[
\tilde F=B\,d\mathrm{vol}_{\mathbb H}=B\frac{dx\wedge dy}{y^2}.
\]

Let $\Gamma\subset PSL(2,\mathbb R)$ be a cocompact lattice acting freely on $\mathbb H$. Then $M=\Gamma\setminus\mathbb H$ is a compact Riemannian surface. The quantization condition \eqref{e:prequantum} holds if
\begin{equation}\label{e:hyper-quant}
(2g-2)B\in \mathbb Z,
\end{equation}
where $g$ is the genus of the surface $M$.

Indeed (see, for instance, \cite{Bolte-Steiner91}), if $F\in H^2(M,2\pi\mathbb Z)$, then
\[
\int_MF=B\, \mathrm{Vol}(M)\in 2\pi\mathbb Z.
\]
On the other hand, by the Gauss-Bonnet theorem, we have
\[
\int_MK\,d\mathrm{vol}_M=2\pi\chi(M)=2\pi(2-2g),
\]
where $K=-1$ is the Gauss curvature of $M$. Hence, we get
\[
\mathrm{Vol}(M)=2\pi(2g-2).
\]

Sections of the associated Hermitian line bundle $L$ on $M$ can be identified with functions $\psi$ on $\mathbb H$, satisfying the condition
\begin{equation}\label{e:sections}
\psi(\eta z)=\psi(z)\exp(-i 2B\arg(cz+d))=\left(\frac{cz+d}{|cz+d|}\right)^{-2B}\psi(z)
\end{equation}
for any $z\in \mathbb H$ and $\eta=\begin{pmatrix} a & b\\ c & d\end{pmatrix}\in \Gamma$.

Let $K=T^*_{(1,0)}M$ be the canonical bundle of the Riemann surface $M$. With a function $\psi : \mathbb H\to \mathbb C$, satisfying \eqref{e:sections}, one can associate a section of $K^{-B}$ in the following way. Put $f(z):= y^B\psi(z)$. Then it is easy to see that
\begin{equation}\label{e:automorphic}
f(\eta z) = (cz+d)^{-2B}f(z), \quad \eta\in \Gamma.
\end{equation}
Since $d(\eta z)/dz = (cz+ d)^{-2}$, we have 
\[
f(\eta z)d(\eta z)^{-B}= f(z)dz^{-B}
\] 
and, therefore, $f$ defines a section of $K^{-B}$. The function $\psi$ itself defines a section of $K^{-B/2}\otimes \bar{K}^{B/2}$. Holomorphic functions $f$, satisfying \eqref{e:automorphic}, are called $\Gamma$-automorphic forms of weight $-2B$. 

The corresponding magnetic Laplacian $\Delta^L$ first appeared in the theory of automorphic forms, where it is known as the Maass Laplacian. Its spectrum was computed in \cite{Elstrodt} (see also \cite{Comtet-Houston85, FV}). In \cite{Br-Nekr-Shaf}, the authors used a semiclassical approach based on the canonical Maslov operator theory to describe the spectrum of the magnetic Laplacian $\Delta^L$ and its degeneration (see also a survey \cite{dobro-sh}).

\begin{thm}\label{t:hyper}
Let $B=1$. For any $\varphi\in \mathcal S(\RR)$ and $1<E<\sqrt{1/R^2+1}$, one has an asymptotic expansion
\begin{equation}\label{e:4.1}
Y_N(\varphi) \sim \sum_{j=0}^\infty c_j(N,\varphi)N^{1-j}, \quad N\to \infty,
\end{equation}
The coefficients $c_j$ can be computed explicitly. For the first two of them, we get
\begin{multline}\label{e:c0-hyper}
c_0(N,\varphi) = (2g-2)ER^2 \sum_{k\in \ZZ} \hat\varphi \left(\frac{2\pi kER}{\sqrt{\frac{1}{R^2}+1-E^2}}\right)\\ \times \exp(ik\pi)\exp\left(2\pi ik R\sqrt{\frac{1}{R^2}+1-E^2}N\right).
\end{multline}
\begin{equation*}
\begin{aligned}
c_1(N,\varphi) =& \Bigg[\sum_{k\in \ZZ}(2g-2)2iR^2\hat\varphi^\prime \left(\frac{2\pi kER}{\sqrt{\frac{1}{R^2}+1-E^2}}\right) \\
&+\sum_{k\in \ZZ} (2g-2)\frac{\pi ikER}{4\sqrt{\frac{1}{R^2}+1-E^2}}\hat\varphi \left(\frac{2\pi kER}{\sqrt{\frac{1}{R^2}+1-E^2}}\right)\\
& +\sum_{k\in \ZZ} (2g-2)i\frac{\pi ikE(R^2+1)R}{(\frac{1}{R^2}+1-E^2)^{3/2}}\hat\varphi^{\prime\prime} \left(\frac{2\pi kER}{\sqrt{\frac{1}{R^2}+1-E^2}}\right)\Bigg] \\
&  \times \exp(ik\pi)\exp\left(2\pi ik R\sqrt{\frac{1}{R^2}+1-E^2}N\right).
\end{aligned}
\end{equation*}
\end{thm}

\begin{proof}
For an arbitrary $B$, the magnetic Laplacian on $\mathbb H$ is given by
\[
\Delta^B=-\frac{y^2}{R^2}\left(\left(\frac{\partial}{\partial x}-iBy^{-1}\right)^2+\frac{\partial^2}{\partial y^2}\right).
\]
Its spectrum on $L^2(\mathbb H)$ consists of two parts: eigenvalues of infinite multiplicity in the interval $[0,\left(1/4+B^2\right)/R^2)$ of the form (hyperbolic Landau levels)
\begin{multline}\label{e:integrable-part}
E_j=\frac{1}{R^2}\left(\frac 14+B^2-\left(j+\frac 12-|B|\right)^2\right) \\ =\frac{1}{R^2}\left((2j+1)|B|-j(j+1)\right), \quad 0\leq j<|B|-\frac 12,
\end{multline}
and the semi-axis $[\left(1/4+B^2\right)/R^2,+\infty)$.

Therefore, if $B$ satisfies the quantization condition, then the spectrum of the corresponding magnetic Laplacian $\Delta^L$ on $M$ consists of two parts --- integrable and chaotic. We are interested in the integrable part of the spectrum, which consists of the eigenvalues $E_j$ given by \eqref{e:integrable-part}. By the Selberg formula, each $E_j$ has multiplicity 
\[
m_j=(g-1)(2|B|-2j-1), \quad 0\leq j<|B|-\frac 12.
\]

Put $B=1$. Then for any $N\in \NN$ the integrable part of the spectrum of the magnetic Laplacian $\Delta^{L^N}$ consists of the eigenvalues
\[
\nu_{N,j} =\frac{1}{R^2}\left(\frac 14+N^2-\left(j+\frac 12-N\right)^2\right), \quad 0\leq j<N-\frac 12
\]
with multiplicity
\[
m_{N,j}=(g-1)(2N-2j-1), \quad 0\leq j<N-\frac 12.
\]
Observe that the associated numbers $\lambda_{N,j}$ defined by \eqref{e:def-lambda}:
\[
\lambda_{N,j}=\sqrt{N^2+\frac{1}{R^2}\left(\frac 14+N^2-\left(j+\frac 12-N\right)^2\right)}, \quad 0\leq j<N-\frac 12,
\]
satisfy the condition 
\[
|N|<\lambda_{N,j}<\sqrt{N^2+\frac{1}{R^2}\left(N^2+\frac14\right)}.
\] 
Therefore, we get the following expression for the function $Y_N(\varphi)$ for $\varphi\in \mathcal S(\RR)$ and $1<E<\mu:=\sqrt{R^2+1}/R$:
\begin{multline} \label{e:hyper-Yp}
Y_N(\varphi)=\sum_{0\leq j<N-\frac 12}(g-1)(2N-2j-1) \times \\ \times \varphi\left(\sqrt{\mu^2N^2+\frac{1}{4R^2}-\frac{1}{R^2}\left(j+\frac 12-N\right)^2}-EN\right).
\end{multline}
Let us make the change of the summation index
in the right hand side of \eqref{e:hyper-Yp} by the formula
\[
j=(1-R\sqrt{\mu^2-E^2})N-\{(1-R\sqrt{\mu^2-E^2})N\}+n,  \quad n\in \ZZ,
\]
where
\[
-[(1-R\sqrt{{\mu^2-E^2}})N] \leq n<N-[(1-R\sqrt{{\mu^2-E^2}})N]-\frac 12.
\]
To simplify the notation, put 
\[
P=\frac{\sqrt{\mu^2-E^2}}{ER}, \quad t=\frac{\sqrt{\mu^2-E^2}}{ER}\left(\frac 12-\{(1-R\sqrt{\mu^2-E^2})N\}\right).
\]
Then the numbers $\lambda_{N,j}^2$ are of the form
\[
\lambda_{N,j}^2=E^2N^2+2\frac{\sqrt{\mu^2-E^2}}{R}\left(n+\frac{t}{P}\right)N+\frac {1}{4R^2}-\frac{1}{R^2}\left(n+\frac{t}{P}\right)^2,
\]
and, by \eqref{e:Taylor-f}, we get the asymptotic expansion for the argument of $\varphi$ in \eqref{e:hyper-Yp} as  $N\to \infty$:
\[
\lambda_{N,j}-EN=nP+t+\left[\frac{1}{8ER^2}-\frac{\mu^2}{2(\mu^2-E^2)E}\left(nP+t\right)^2\right]N^{-1}+\sum_{k=2}^\infty a_k(N)N^{-k}.
\]
For the multiplicities, we get
\[
m_{N,n}=(2g-2)\left(R\sqrt{\mu^2-E^2}N-n-\frac{t}{P}\right).
\]
As above, using these asymptotic expansions and \eqref{e:Taylor-phi}, we obtain the asymptotic expansion \eqref{e:4.1}.

The computation of the coefficients $_0$ and $c_1$ in the expansion \eqref{e:4.1} can be made by analogy with the computation given in Subsection~\ref{s:torus}, using the Poisson summation formula \eqref{e:Poisson}, and we will omit it. Theorem~\ref{t:hyper} is proved.
\end{proof}

We give a geometric interpretation of the formula  \eqref{e:c0-hyper},  comparing it with the general formulas \eqref{e:c0-gen0} and \eqref{e:c0-gen1}.

The Hamiltonian $H$ defined by \eqref{e:H} has the form
\[
H(x,y,p_x,p_y)=\left(\frac{1}{R^2}y^2(p_x^2+p_y^2)+1\right)^{1/2}.
\]
Therefore, we have
\[
X_E=H^{-1}(E)=\{(x,y,p_x,p_y) : y^2(p_x^2+p_y^2)=(E^2-1)R^2\}\subset T^*M,
\]
and the Liouville measure on $X_E$ has the form:
\[
{\rm vol}=\frac{dx\,dy\,dp_x\,dp_y}{dH}.
\]
Hence,
\[ 
{\rm Vol}(X_E)= 2\pi E {\rm Vol}(M) =(2\pi)^2 (2g-2)ER^2.
\]
This shows that \eqref{e:c0-hyper} for $k=0$ agrees with \eqref{e:c0-gen0}.

Recall a well-known description of the magnetic geodesic flow on the hyperbolic plane (see, for instance, \cite[p. 416]{Gu-Uribe89}).

If $cR=\sqrt{E^2-1}R>1$, then the projections of trajectories of the flow on  $X_E$ to $\mathbb H$ are circular arcs that intersect the $x$-axis in the angle $\theta$ such that $1/\cos\theta=cR$. The flow is conjugate to the geodesic flow on $X_E$ up to a constant time scaling.

If $cR=\sqrt{E^2-1}R=1$, then the projections of trajectories of the flow on  $X_E$ to $\mathbb H$ are classical horocycles.

If $cR=\sqrt{E^2-1}R<1$, then the projections of trajectories of the flow on  $X_E$ to $\mathbb H$ are circles on the hyperbolic plane of radius $r$, where $\tanh r=cR$. The length of such a circle equals
\[
L=2\pi \sinh(r)=\frac{2\pi cR}{\sqrt{1-c^2R^2}}.
\]
The flow is periodic with period $T=2\pi E/\sqrt{1-c^2R^2}$. So it is clean on $X_E$ and $d=n-1=2$.

This implies, in particular, that at the level $E=\sqrt{R^2+1}/R$, i.e. for $\cE = 1/(2R)$, the magnetic geodesic flow on the closed Riemann surface $M$ has no periodic trajectories \cite{Ginzburg}.

Let us compute $h_{A}(\gamma)\in S^1$, the holonomy of $\pi_M\circ \gamma$ with respect to the connection $\nabla^{L}$:
\[
h_{A}(\gamma)=\int_\gamma A=\int_\gamma \frac{1}{y}dx.
\]
By Stokes formula, we have:
\[
\int_\gamma A=\int_D dx\wedge dy =S
\]
(here $D$ is the disc on the hyperbolic plane, bounded by the circle $\gamma$, and $S$ is its area). On the other hand, the area of the geodesic disc equals 
\[
S=2\pi (\cosh(r)-1)=2\pi\left(\frac{1}{\sqrt{1-c^2R^2}}-1\right).
\]
Thus, we get
\[
\int_\gamma A=2\pi\left(\frac{1}{\sqrt{1-c^2R^2}}-1\right).
\]
Finally, by \eqref{e:Sgamma}, we find
\[
S_\gamma=-\int_\gamma A+L\sqrt{E^2-1}=2\pi(1-\sqrt{1-c^2R^2}).
\]
This shows that \eqref{e:c0-hyper} for $k\neq0$ agrees with \eqref{e:c0-gen1}.

\subsection{The Katok example}
Suppose that the manifold $M$ is the two-sphere
\[
x^2+y^2+z^2=1.
\]
In the spherical coordinates
\[
x=\sin\theta \cos\varphi, \quad y=\sin\theta \sin\varphi, \quad z=\cos\theta, \quad \theta\in (0,\pi), \varphi\in (0,2\pi),
\]
the Riemannian metric $g$ on $M$ is given by
\[
g=\frac{d\theta^2}{1-\varepsilon^2\sin^2\theta}+\frac{\sin^2\theta}{(1-\varepsilon^2\sin^2\theta)^2}d\varphi^2,
\]
and the magnetic field form $F$ by
\[
F=dA=\frac{\varepsilon \sin 2\theta}{(1-\varepsilon^2\sin^2\theta)^2}d\theta\wedge d\varphi,
\]
where
\[
A=-\frac{\varepsilon \sin^2\theta}{1-\varepsilon^2\sin^2\theta}d\varphi.
\]
Since the magnetic field form is exact, the quantization condition holds and the quantum line bundle $L$ is trivial.

In the first order of perturbation theory, i.e. modulo $\varepsilon^2$, this system describes the motion of a charged particle on the sphere of radius  $1$ in $\RR^3$ in the external magnetic field 
\[
F=\frac{1}{2}\varepsilon\, dx \wedge dy.
\]

\begin{thm}\label{t:8}
Let $\varphi\in \mathcal S(\RR)$, $E=\sqrt{2}$ and $\varepsilon\in (0,1)$ is irrational. If $0$ is a unique period in $\supp(\hat\varphi)$, then one has an asymptotic expansion
\[
Y_N(\varphi) \sim \sum_{j=0}^\infty c_j(N,\varphi)N^{1-j}, \quad N\to \infty,
\]
where
\begin{equation}\label{e:c0-Katok1}
c_0(N,\varphi)=2\sqrt{2}\hat{\varphi}(0).
\end{equation}

If $0$ is not in the support of $\hat\varphi$, then one has an asymptotic expansion
\[
Y_N(\varphi) \sim \sum_{j=0}^\infty c_j(N,\varphi)N^{-j}, \quad N\to \infty,
\]
where
\begin{multline}\label{e:c0-Katok2}
c_0(N,\varphi)=\sum_{k\neq 0} \frac{1}{\sqrt{2}(1-\varepsilon^2)} \\ \times  \left(\frac{e^{ik\pi m_{k,+}/4} e^{-iNk\frac{2\pi}{1-\varepsilon}}}{\sin \frac{\pi k}{1-\varepsilon}}+\frac{e^{ik\pi m_{k,-}/4} e^{-iNk\frac{2\pi}{1+\varepsilon}}}{\sin \frac{\pi k}{1+\varepsilon}}\right)\hat\varphi\left(\frac{2\pi \sqrt{2}k}{1-\varepsilon^2}\right),
\end{multline}
the Maslov indices $m_{k,\pm}$ are given by
\begin{equation}
\label{e:m-Katok}
m_{k,\pm} = 2\left[\frac{2k}{1\mp \varepsilon}\right]+2\, {\rm sign}\, k+1.
\end{equation}
\end{thm}

\begin{proof}
If $0$ is a unique period in $supp(\hat\varphi)$, then, by \eqref{e:c0-gen0}, we have
\[
c_0(N,\varphi)=(2\pi)^{-2}\hat{\varphi}(0){\rm Vol}(X_E),
\]
As in the case of the two-sphere, one can compute
\[
{\rm Vol}(X_E)= 2\pi E{\rm Vol}(M)=8\pi^2 E,
\]
that gives \eqref{e:c0-Katok1}.

The Hamiltonian $H$ given by \eqref{e:H} has the form
\[
H=\left((1-\varepsilon^2\sin^2\theta)p_\theta^2+\frac{(1-\varepsilon^2\sin^2\theta)^2}{\sin^2\theta}p_\varphi^2+1\right)^{1/2}.
\]
The reduced Hamiltonian system on $X_E$ is given by
\begin{equation}\label{e:Ham-Katok}
\begin{aligned}
\dot{\theta}=& \frac{1}{E}(1-\varepsilon^2\sin^2\theta)p_\theta, \\ \dot{\varphi}= & \frac{1}{E}\frac{(1-\varepsilon^2\sin^2\theta)^2}{\sin^2\theta}p_\varphi,\\
\dot{p}_\theta = & \frac{\varepsilon^2}{E}\sin\theta\cos\theta p_\theta^2+\frac{1}{E}\left[\frac{\cos\theta}{\sin^3\theta}-\varepsilon^4\sin \theta\cos \theta\right]p_\varphi^2+\frac{2\varepsilon}{E}\cot \theta p_\varphi,\\
\dot{p}_\varphi=& -\frac{\varepsilon}{E}\frac{\sin 2\theta}{1-\varepsilon^2\sin^2\theta}p_\theta.
\end{aligned}
\end{equation}
This system is integrable with an additional first integral 
\[
P=p_\varphi+\varepsilon \frac{\sin^2\theta}{1-\varepsilon^2\sin^2\theta}.
\]
It is easy to see that the system \eqref{e:Ham-Katok} has two particular periodic solutions, corresponding to the equator of the sphere, running in different directions:
\begin{equation}\label{e:lpm}
\begin{split}
\theta(t)=&\frac{\pi}{2}, \quad \varphi(t)=\pm \frac{\sqrt{E^2-1}}{E}(1-\varepsilon^2)t+\varphi_0, \\ p_\theta(t)=& 0, \quad p_\varphi(t)=\pm \frac{\sqrt{E^2-1}}{1-\varepsilon^2}
\end{split}
\end{equation}
with period
\begin{equation}\label{e:T-lpm}
T=\frac{2\pi E}{(1-\varepsilon^2)\sqrt{E^2-1}}.
\end{equation}
The length of the corresponding periodic trajectory $\pi_M\circ \ell_{\pm}$ on $M$ equals
\[
L=\frac{2\pi}{1-\varepsilon^2}.
\]

It turns out that, if $E=\sqrt{2}$ and $\varepsilon$ is irrational, then these trajectories are the only periodic trajectories of the magnetic geodesic flow. Indeed, first of all, observe that the restriction of the Hamiltonian system to the energy level $X_E$ is described by the Lagrangian 
\[
L(\theta,\varphi,\dot\theta,\dot\varphi)=\sqrt{E^2-1}\frac{\sqrt{(1-\varepsilon^2\sin^2\theta)\dot\theta^2+\sin^2\theta\dot\varphi^2}}{(1-\varepsilon^2\sin^2\theta)^2}-\frac{\varepsilon \sin^2\theta}{1-\varepsilon^2\sin^2\theta}\dot\varphi.
\]
Therefore, this statement follows from Theorem 5 in \cite{Rademacher04}, where this Lagrangian is shown to be exactly a Finsler metric introduced by A. B. Katok in \cite{Katok}. As observed in \cite{Rademacher04}, this metric was mentioned in \cite{Shen} as an example of constant flag curvature Finsler metric on $S^2$.

Let us show that for an irrational $\varepsilon$ the periodic trajectories $\ell_{\pm}$ are non-degenerate. In a neighborhood of $\ell_\pm$, we can choose the variables $(\varphi, \theta, p_\theta)$ to be coordinates on $X_E$. Then the tangent space $T_{(\varphi, \theta, p_\theta)}X_E$ of $X_E$ at a point with coordinates $(\varphi, \theta, p_\theta)$ consists of all 
\[
(\Phi, \Theta, P_\varphi, P_\theta)\in T_{(\varphi, \theta, p_\varphi, p_\theta)}(T^*S^2)
\] 
such that
\begin{multline*}
\left[-\varepsilon^2\cdot  2\sin\theta\cos\theta p_\theta^2+\left(-\frac{2\cos\theta}{\sin^3\theta}+2\varepsilon^4\sin \theta\cos \theta\right)p_\varphi^2\right]\Theta \\ + 2(1-\varepsilon^2\sin^2\theta)p_\theta P_\theta+2\frac{(1-\varepsilon^2\sin^2\theta)^2}{\sin^2\theta}p_\varphi P_\varphi=0.
\end{multline*}
In particular, if $(\varphi, \theta, p_\theta)\in \ell_\pm$, then $(\Phi, \Theta, P_\varphi, P_\theta)\in T_{(\varphi, \theta, p_\theta)}X_E$ if and only if $P_\varphi=0$.

The system of variational equations for the system \eqref{e:Ham-Katok} along the solutions \eqref{e:lpm} has the form
\begin{gather*}
\dot{\Theta}=\frac{1}{E}(1-\varepsilon^2)P_\theta,\quad
\dot{\Phi}=\frac{1}{E}(1-\varepsilon^2)^2P_\varphi,\\
\dot{P}_\theta = -\left[\frac{1}{E}\frac{(E^2-1)(1+\varepsilon^2)}{1-\varepsilon^2}\pm\frac{2\varepsilon}{E} \frac{\sqrt{E^2-1}}{1-\varepsilon^2} \right] \Theta,\quad \dot{P}_\varphi=0.
\end{gather*}
This system can be easily solved:
\begin{equation}\label{e:sol-var}
\begin{aligned}
\Theta(t)=& \Theta(0)\cos \frac{a_{\pm}}{E}t+ \frac{1-\varepsilon^2}{a_{\pm}}P_\theta(0) \sin \frac{a_{\pm}}{E}t, \\ \Phi(t)=& \frac{1}{E}(1-\varepsilon^2)^2P_\varphi(0)t+\Phi(0),\\
P_\theta(t)=& -\Theta(0)\frac{a_{\pm}}{1-\varepsilon^2} \sin \frac{a_{\pm}}{E}t  + P_\theta(0) \cos \frac{a_{\pm}}{E}t, \\ P_\varphi(t)=& P_\varphi(0),
\end{aligned}
\end{equation}
where we put
\[
a_{\pm}=\left[(E^2-1)(1+\varepsilon^2)\pm 2\varepsilon \sqrt{E^2-1}\right]^{1/2}.
\]
For $(\varphi, \theta, p_\theta)\in \ell_\pm$, the tangent space $T_{(\varphi, \theta, p_\theta)}\ell_{\pm}$ consists of all 
\[
(\Phi, \Theta, 0, P_\theta)\in T_{(\varphi, \theta, p_\theta)}X_E
\] 
such that $\Theta=0, P_\theta=0$. Therefore, one can take $(\Theta, P_\theta)$ to be coordinates in the normal space $T_{(\varphi, \theta, p_\theta)}X_E/T_{(\varphi, \theta, p_\theta)}\ell_{\pm}$. The Poincar\'e map $P_{\ell_\pm}$ takes an initial point $(\Theta(0), P_\theta(0))$ to the solution $(\Theta(t), P_\theta(t))$ given by \eqref{e:sol-var} evaluated at a period $t=T$, where $T$ is given by \eqref{e:T-lpm}. We have:
\begin{align*}
\Theta(T)=& \Theta(0)\cos\alpha_\pm + \frac{1-\varepsilon^2}{a_{\pm}}P_\theta(0) \sin\alpha_\pm, \\
P_\theta(T)=& -\Theta(0)\frac{a_{\pm}}{1-\varepsilon^2} \sin\alpha_\pm  + P_\theta(0) \cos\alpha_\pm,
\end{align*}
where
\[
\alpha_\pm= \frac{2\pi }{1-\varepsilon^2} \left[1+\varepsilon^2\pm\frac{2\varepsilon}{\sqrt{E^2-1}}\right]^{1/2}.
\]
Thus, the Poincar\'e map $P_{\ell_\pm}$ is given by the $(2\times 2)$-matrix
\[
P_{\ell_\pm}=
\begin{pmatrix}
\cos\alpha_\pm & \frac{1-\varepsilon^2}{a_{\pm}} \sin\alpha_\pm \\
-\frac{a_{\pm}}{1-\varepsilon^2} \sin\alpha_\pm & \cos\alpha_\pm
\end{pmatrix}.
\]
In particular, we have
\[
|I-P_{\ell_\pm}|=4\sin^2 \frac{\alpha_\pm}{2}.
\]
For $E=\sqrt{2}$, the equalities hold:
\[
T^\#_\gamma=\frac{2\pi \sqrt{2}}{1-\varepsilon^2},\quad \alpha_\pm = \frac{2\pi k}{1\mp\varepsilon}.
\]
Hence, if $\varepsilon$ is irrational, the periodic trajectories $\ell_{\pm}$  are non-degenerate. Thus, if $0$ is not in the support of $\hat\varphi$, then $d = 0$, and we can use the formula \eqref{e:c0-nondegenerate}.

Let us compute the action of $\ell_{\pm}$. By Stokes formula, we get
\begin{align*}
\int_{\gamma_\pm} A= & \pm \iint_{S^2_+} \frac{\varepsilon \sin 2\theta}{(1-\varepsilon^2\sin^2\theta)^2}d\theta\wedge d\varphi\\
= & \pm 2\pi \varepsilon \int_{0}^{\pi/2} \frac{\sin 2\theta}{(1-\varepsilon^2\sin^2\theta)^2}d\theta=  \pm \frac{2\pi \varepsilon}{1-\varepsilon^2}.
\end{align*}
By \eqref{e:Sgamma}, it follows that
\[
S_\gamma=\int_\gamma A+L\sqrt{E^2-1}=\pm \frac{2\pi \varepsilon}{1-\varepsilon^2}+\frac{2\pi}{1-\varepsilon^2}\sqrt{E^2-1}.
\]
In particular, for $E=\sqrt{2}$, we have 
\[
S_\gamma=\frac{2\pi}{1\mp \varepsilon}.
\]

Let us compute the Maslov index $m_\gamma$ for the periodic trajectory with period 
\[
T=kT^\#_\gamma=\frac{2\pi E}{1-\varepsilon^2},\quad k\in \ZZ.
\] 
Since the form $F$ is exact, the Hermitian line bundle $L$ is trivial as well as the orthonormal frame bundle $S$: 
\[
L=S^2\times \CC, \quad S=S^2\times S^1.
\] 
In the coordinates $(\theta, \varphi, \chi) \in (0,\pi)\times (0,2\pi)\times (0,2\pi)$ on $S$, the connection form $\alpha : TS\to \RR$ has the form
\[
\alpha(\theta, \varphi, \chi)=d\chi+\frac{\varepsilon \sin^2\theta}{1-\varepsilon^2\sin^2\theta}d\varphi.
\]
The horizontal distribution of the connection is spanned by the vectors
\[
\frac{\partial}{\partial \theta}, \quad \frac{\partial}{\partial \varphi}-\frac{\varepsilon \sin^2\theta}{1-\varepsilon^2\sin^2\theta} \frac{\partial}{\partial \chi}.
\]
The vertical distribution is spanned by $\partial/\partial \chi$. The Riemannian metric $g_S$ on $S$ is given by
\[
g_S=\frac{d\theta^2}{1-\varepsilon^2\sin^2\theta}+\frac{\sin^2\theta}{(1-\varepsilon^2\sin^2\theta)^2}d\varphi^2+\left(d\chi+\frac{\varepsilon \sin^2\theta}{1-\varepsilon^2\sin^2\theta}d\varphi\right)^2.
\]

The Hamiltonian of the geodesic flow $f_t$ on $T^*S$ has the form
\[
h=\left((1-\varepsilon^2\sin^2\theta)p_\theta^2+\frac{(1-\varepsilon^2\sin^2\theta)^2}{\sin^2\theta}\left(p_\varphi -\frac{\varepsilon \sin^2\theta}{1-\varepsilon^2\sin^2\theta}p_\chi\right)^2+p^2_\chi\right)^{1/2}.
\]
Therefore, for $E>1$, the restriction of the geodesic flow to the submanifold
\[
W=\{(\theta, \varphi, \chi, p_\theta, p_\varphi, p_\chi)\in T^*S : h(\theta, \varphi, \chi, p_\theta, p_\varphi, p_\chi)=E, p_\chi=1\}.
\]
is given by the system of equations 
\begin{equation}\label{e:Ham-Katok1}
\begin{aligned}
\dot{\theta}=& \frac{1}{E}(1-\varepsilon^2\sin^2\theta)p_\theta, \\ \dot{\varphi}= & \frac{1}{E}\left[\frac{(1-\varepsilon^2\sin^2\theta)^2}{\sin^2\theta}p_\varphi - \varepsilon(1-\varepsilon^2\sin^2\theta)p_\chi \right],\\
\dot{\chi}=& \frac{1}{E}\left[-\varepsilon(1-\varepsilon^2\sin^2\theta)p_\varphi  + (1+\varepsilon^2\sin^2\theta)p_\chi \right], \\
\dot{p}_\theta = & \frac{\varepsilon^2}{E}\sin\theta\cos\theta p_\theta^2+\frac{1}{E}\left[\frac{\cos\theta}{\sin^3\theta}-\varepsilon^4\sin \theta\cos \theta\right]p_\varphi^2\\
& - \frac{2}{E} \varepsilon^3 \sin\theta \cos \theta p_\varphi p_\chi -\frac{1}{E}  \varepsilon^2 \sin\theta \cos \theta p^2_\chi \\
\dot{p}_\varphi=& 0, \\
\dot{p}_\chi=& 0.
\end{aligned}
\end{equation}

It is easy to see that the system has two particular solutions, corresponding to the equatorial closed magnetic geodesics:
\begin{equation}\label{e:lpm1}
\begin{split}
\theta(t)=&\frac{\pi}{2}, \quad \varphi(t)=\pm \frac{\sqrt{E^2-1}}{E}(1-\varepsilon^2)t+\varphi_0,\\
\chi(t) = & \frac{1}{E}(1\mp \varepsilon \sqrt{E^2-1})t+\chi_0 ,
 \\ p_\theta(t)=& 0, \quad p_\varphi(t)=\pm \frac{\sqrt{E^2-1}}{1-\epsilon^2}+\frac{\varepsilon}{1-\varepsilon^2}, \quad p_\chi(t)=1.
\end{split}
\end{equation}
The system of variational equations for the system \eqref{e:Ham-Katok1} along the solutions \eqref{e:lpm1} has the form 
\begin{align*}
\dot{\Theta}=& \frac{1}{E}(1-\varepsilon^2)P_\theta,\\
\dot{\Phi}=& \frac{1}{E}(1-\varepsilon^2)^2 P_\varphi -\frac{1}{E}\varepsilon (1-\varepsilon^2)P_\chi,\\
\dot{X}= & \frac{1}{E}\left[-\varepsilon(1-\varepsilon^2)P_\varphi  + (1+\varepsilon^2)P_\chi \right],\\
\dot{P}_\theta = &-\left[\frac{1}{E}\frac{(E^2-1)(1+\varepsilon^2)}{1-\varepsilon^2}\pm\frac{2\varepsilon}{E} \frac{\sqrt{E^2-1}}{1-\varepsilon^2} \right] \Theta,\\
\dot{P}_\varphi=& 0,
\\ \dot{P}_\chi=& 0.
\end{align*}
This system can be easily solved:
\begin{equation}\label{e:sol-var1}
\begin{aligned}
\Theta(t)=& \Theta(0)\cos \frac{a_{\pm}}{E}t+ \frac{1-\varepsilon^2}{a_{\pm}}P_\theta(0) \sin \frac{a_{\pm}}{E}t, \\ \Phi(t)=& \frac{1}{E}\left[(1-\varepsilon^2)^2 P_\varphi(0)-\varepsilon (1-\varepsilon^2)P_\chi(0)\right]t+\Phi(0),\\
X(t)= & \frac{1}{E}\left[-\varepsilon(1-\varepsilon^2)P_\varphi(0)  + (1+\varepsilon^2)P_\chi(0) \right]t+X(0),\\
P_\theta(t)=& -\Theta(0)\frac{a_{\pm}}{1-\varepsilon^2} \sin \frac{a_{\pm}}{E}t  + P_\theta(0) \cos \frac{a_{\pm}}{E}t, \\ P_\varphi(t)=& P_\varphi(0), \\ P_\chi(t)=& P_\chi(0).
\end{aligned}
\end{equation}

The vertical space $V_{f^t(\nu)}$ for $\nu=(\pi/2,\varphi_0,\chi_0$, $0$,$p_\varphi(0)$, $1)$ is given by the relations $\Theta(t)=0, \Phi(t)=0, X(t)=0$. Therefore, for any $0\leq t\leq T$, the Lagrangian subspace $\omega(t)=df^{-t}(V_{f^t(\nu)})$ in $T_\nu(T^*S)$ consists of vectors of the form
\begin{align*}
\Theta(t) =& -\frac{1-\varepsilon^2}{a_{\pm}} P_\theta \sin \frac{a_{\pm}}{E}t, \\ \Phi(t)=& -\frac{1}{E}\left[(1-\varepsilon^2)^2 P_\varphi-\varepsilon (1-\varepsilon^2)P_\chi\right]t,\\
X(t)= & -\frac{1}{E}\left[-\varepsilon(1-\varepsilon^2)P_\varphi  + (1+\varepsilon^2)P_\chi \right]t,\\
P_\theta(t)=& P_\theta \cos \frac{a_{\pm}}{E}t, \\ P_\varphi(t)=& P_\varphi, \\ P_\chi(t)=& P_\chi
\end{align*}
with arbitrary $P_\theta, P_\varphi, P_\chi$.

It is easy to see that the intersection number $\kappa=[\omega:H_\nu]$ of the curve $\omega$ with the horizontal subspace $H_\nu\in T_\nu(T^*S)$ equals the intersection number of the curve $\tilde\omega(t), 0\leq t\leq T,$ of Lagrangian subspaces of the space $\RR^2$ with coordinates $(\Theta, P_\theta)$, obtained by rotation of the vertical subspace $\Theta=0$ by the angle $(a_{\pm}/E)t$, with the horizontal subspace $P_\theta=0$. Therefore, $\kappa$ is the integer such that the angle 
\[
\frac{a_{\pm}}{E}T=\frac{2\pi k}{1\mp \varepsilon}
\] 
is in the interval $((2\kappa-1)\pi/2, (2\kappa+1)\pi/2)$ that implies that
\[
\kappa=\begin{cases} \ell, & \ \text{if}\ \ell \in \mathbb Z, \ell < \frac{2k}{1\mp \varepsilon}< \ell+\frac 12, \\
\ell+1, & \ \text{if}\ \ell \in \mathbb Z, \ell+\frac 12< \frac{2k}{1\mp \varepsilon}< \ell+1.
\end{cases}
\]

The generating function of the linear symplectic map $df^T$ given by  \eqref{e:sol-var1} has the form
\begin{multline*}
F(\Theta,\Phi,X,P_\theta,P_\varphi,P_\chi)=\Phi P_\varphi +XP_\chi \\
\begin{aligned}
&+\frac{1}{E}\left[\frac 12(1-\varepsilon^2)^2 P^2_\varphi-\varepsilon (1-\varepsilon^2)P_\varphi P_\chi+ \frac 12(1+\varepsilon^2)P^2_\chi\right]T \\
& +\frac{a_{\pm}}{2(1-\varepsilon^2)}\Theta^2 \tan \frac{a_{\pm}}{E}T +\frac{1}{\cos \frac{a_{\pm}}{E}T}\Theta P_\theta +\frac{1-\varepsilon^2}{2a_{\pm}}P^2_\theta \tan \frac{a_{\pm}}{E}T.
\end{aligned}
\end{multline*}
For the signature of the matrix $R$ given by \eqref{e:R}, we get
\[
{\rm sgn}\,R=\begin{cases} 1+2\, {\rm sign}\, k, & \ \text{if}\ \ell \in \mathbb Z, \ell<\frac{2k}{1\mp \varepsilon}< \ell+\frac 12, \\
-1+2\, {\rm sign}\, k, & \ \text{if}\ \ell \in \mathbb Z, \ell+\frac 12< \frac{2k}{1\mp \varepsilon}< \ell+1.
\end{cases}
\]
Finally, by \eqref{e:m}, we obtain \eqref{e:m-Katok}.

This completes the computation of \eqref{e:c0-nondegenerate} in this particular case that gives \eqref{e:c0-Katok2}. Theorem~\ref{t:8} is proved. 
\end{proof}

\end{document}